\documentclass[12pt]{amsart}
\usepackage{pifont}
\usepackage{txfonts}
\usepackage{}
\usepackage{amsfonts}
\usepackage{graphicx}
\usepackage{epstopdf}
\usepackage{amsmath}
\usepackage{amsthm}
\usepackage{latexsym,bm}
\usepackage{amssymb}
\usepackage{esint}
\usepackage{enumitem}
\usepackage{indentfirst}
\usepackage{amscd}
\newtheorem{thm}{Theorem}[section]
\newtheorem{cor}[thm]{Corollary}
\newtheorem{lem}[thm]{Lemma}
\newtheorem{prop}[thm]{Proposition}
\newtheorem{defn}[thm]{Definition}

\newtheorem{Remark}{Remark}
\numberwithin{equation}{section}
\numberwithin{Remark}{section}

\begin{document}

\title{A compactness theorem on Branson's $Q$-curvature equation}

\author{Gang Li$^\dag$}

\address{Gang Li, Beijing International Center for Mathematical Research, Peking University, Beijing, China}
\email{runxing3@gmail.com}

\begin{abstract}
Let $(M, g)$ be a closed Riemannian manifold of dimension $5$. Assume that $(M, g)$ is not conformally equivalent to the round sphere. If the scalar curvature $R_g\geq 0$ and the $Q$-curvature $Q_g\geq 0$ on $M$ with $Q_g(p)>0$ for some point $p\in M$, we prove that the set of metrics in the conformal class of $g$ with prescribed constant positive $Q$-curvature is compact in $C^{4, \alpha}$ for any $0 <\alpha < 1$. We also give some estimates for dimension $6$ and $7$.
\end{abstract}


\thanks{$^\dag$ Research supported by China Postdoctoral Science Foundation Grant 2014M550540.}
\maketitle

\section{Introduction}

On a manifold $(M^n, g)$ of dimension $n\geq 5$, the $Q$-curvature of Branson \cite{Branson} is defined by
\begin{align*}
Q_g=-\frac{2}{(n-2)^2}|Ric_g|^2+\frac{n^3-4n^2+16n-16}{8(n-1)^2(n-2)^2}R_g^2-\,\frac{1}{2(n-1)}\Delta_gR_g
\end{align*}
where $Ric_g$ is the Ricci curvature of $g$, $R_g$ is the scalar curvature of $g$ and $\Delta_g$ is the Laplacian operator with negative eigenvalues. The Paneitz operator \cite{Paneitz}, which is the linear operator in the conformal transformation formula of the $Q$-curvature, is defined as
\begin{align}\label{operatorP}
P_g=\Delta_g^2-\text{div}_g(a_nR_g g -b_nRic_g)\nabla_g + \frac{n-4}{2}Q_g,
\end{align}
with $a_n=\frac{(n-2)^2+4}{2(n-1)(n-2)}$ and $b_n=\frac{4}{n-2}$. In fact, under the conformal change $\tilde{g}=u^{\frac{4}{n-4}}g$ the transformation formula of the $Q$-curvature is given by
\begin{align*}
P_gu\,=\,\frac{n-4}{2}Q_{\tilde{g}}u^{\frac{n+4}{n-4}}.
\end{align*}
In comparison, the change of scalar curvature under the conformal change $\tilde{g}=u^{\frac{4}{n-2}}g$ satisfies
\begin{align*}
R_{\tilde{g}}=u^{-\frac{n+2}{n-2}}(-\frac{4(n-1)}{(n-2)})\Delta_g u\,+\,R_g u).
\end{align*}
\vskip0.2cm

Let $(M^n, g)$ be a closed Riemannian manifold of dimension $n\geq 5$. Assume that $R_g\geq 0$ and $Q_g\geq 0$ on $M$ with $Q_g$ not identically zero. For existence of solutions $u$ to the prescribed constant positive $Q$-curvature equation
\begin{align}\label{equation1}
P_gu\,=\,\frac{n-4}{2}\bar{Q}u^{\frac{n+4}{n-4}},
\end{align}
with $\bar{Q}=\frac{1}{8}n(n^2-4)$, one may refer to Qing-Raske \cite{Qing-Raske1}, Hebey-Robert \cite{Hebey-Robert}, Gursky-Malchiodi \cite{Gursky-Malchiodi}, Hang-Yang \cite{Hang-Yang}, Gursky-Hang-Lin \cite{Gursky-Hang-Lin}. Recently, based on a nice maximum principle, Gursky and Malchiodi proved that
\begin{thm}\label{thma}
(Gursky-Malchiodi \cite{Gursky-Malchiodi}) For a closed Riemannian manifold $(M^n, g)$ of dimension $n\geq 5$, if $R_g\geq 0$ and $Q_g\geq 0$ on $M$ with $Q_g$ not identically zero, then there is a conformal metric $h=u^{\frac{4}{n-4}}g$ with positive scalar curvature and constant $Q$-curvature $Q_h=\bar{Q}$.
\end{thm}
Moreover, they showed positivity of the Green's function of the Paneitz operator. Also, for $n=5,\,6,\,7$, they proved a version of positive mass theorem( see Theorem \ref{thm1}), which implies possibility to show compactness of the set of positive solutions to the prescribed constant $Q$-curvature problem in $C^{4, \alpha}(M)$ with $0< \alpha <1$.

For compactness results of solutions to the prescribed constant $Q$-curvature equation with different conditions, one would like to see Djadli-Hebey-Ledoux \cite{Djadli-Hebey-Ledoux}, Hebey-Robert \cite{Hebey-Robert}, Humbert-Raulot \cite{Humbert-Raulot}, Qing-Raske \cite{Qing-Raske}. In Djadli-Hebey-Ledoux \cite{Djadli-Hebey-Ledoux}, the authors studied the optimal Sobolev constant in the embedding $W^{2,2}\hookrightarrow L^{\frac{2n}{n-4}}$ where $P_g$ has constant coefficients. With some additional assumptions, they studied compactness of solutions to the related equations under $W^{2,2}$ bound and obtained existence of positive solutions for the corresponding equations. Under the assumption that the Paneitz operator and positive Green's function, Hebey-Robert \cite{Hebey-Robert} considers compactness of positive solutions with $W^{2,2}$ bound in locally conformally flat manifolds with positive scalar curvature. They showed when the Green's function satisfies a positive mass theorem, the conclusion holds. Later, Humbert-Raulot \cite{Humbert-Raulot} showed that the positive mass theorem holds automatically under the assumption in Hebey-Robert \cite{Hebey-Robert}.  In Qing-Raske \cite{Qing-Raske}, with the use of the developing map and moving plane method, they showed $L^{\infty}$ bound of solutions to the prescribed constant $Q$-curvature equation, for locally conformally flat manifolds with positive scalar curvature without additional assumptions. Combining Qing-Raske's result with positivity of Green's function, one can also easily get the full compactness result, see Theorem \ref{thmconformally_flat}.

In this notes we want to study compactness of solutions to the prescribed constant $Q$-curvature equation in Theorem \ref{thma}, following Schoen's outline of proof of compactness of solutions to the prescribed scalar curvature problem. For compactness results of solutions to the prescribed scalar curvature problem, following Schoen's original outline, one can see Schoen (\cite{Schoen}, \cite{Schoen1}, \cite{Schoen-Zhang}), Li-Zhu \cite{Li-Zhu}, Druet \cite{Druet}, Chen-Lin \cite{Chen-Lin}, Li-Zhang (\cite{Li-Zhang}, \cite{Li-Zhang1}), Marques \cite{Marques}, Khuri-Marques-Schoen \cite{Khuri-Marques-Schoen}. For non-compactness results, see Brendle \cite{Brendle}, Brendle-Marques \cite{Brendle-Marques} and Wei-Zhao \cite{Wei-Zhao}. For compactness argument for the Nirenberg problem for a more general type conformal equation on the round sphere, see Jin-Li-Xiong \cite{Jin-Li-Xiong}. More precisely, we will follow the approach in Li-Zhu \cite{Li-Zhu} for solutions to constant $Q$-curvature problem in dimension $n=5$ under Gursky-Malchiodi's setting. For $n=6, \,7$, we give some estimates along this direction.

Our main theorem is

\begin{thm}\label{thm27}
Let $(M^n, g)$ be a closed Riemannian manifold of dimension $n=5$ with $R_g\geq 0$, and also $Q_g\geq 0$ with $Q_g(p)>0$ for some point $p\in M$. Assume that $(M, g)$ is not conformal equivalent to the round sphere. Then there exists $C>0$ depending on $M$ and $g$ such that for any positive solution to $(\ref{equation1})$,
\begin{align*}
C^{-1}\leq u \leq C,
\end{align*}
and for any $0<\alpha <1$, there exists $C'>0$ depending on $M$, $g$, $\alpha$ and $n$ such that
\begin{align*}
\|u\|_{C^{4, \alpha}}\leq C'.
\end{align*}
\end{thm}

We will perform a contradiction argument between local information from a Pohozaev type identity relating to the constant $Q$-curvature equation and a global discussion provided by the positive mass theorem in Gursky-Malchiodi \cite{Gursky-Malchiodi}( see Theorem \ref{thm1}). In comparison, for compactness of Yamabe problem, the application of positive mass theorem by Schoen and Yau \cite{Schoen-Yau} is crucial.

We give a direct modification of the maximum principle in Gursky-Malchiodi \cite{Gursky-Malchiodi} for manifolds with boundary, see Lemma \ref{lemmaximumprinciple}. It turns out to be very useful 
 in the proof of lower bound of the solutions away from the isolated blowup points( see Theorem \ref{thmlowerbound}) and it plays a role of local maximum principle in estimating upper bounds of solutions near blowup points( see Lemma \ref{lemupperboundestimates1}). To show upper bound of the solutions, we give upper bound estimates of a sequence of blowup solutions near isolated simple blowup points as in Li-Zhu \cite{Li-Zhu}, see in Section \ref{section6}. We are able to prove a Harnack type inequality near the isolated blowup points for $5\leq n \leq 9$, see Lemma \ref{lemH}. Besides the prescribed $Q$-curvature equation, nonnegativity of scalar curvature is also important in the analysis of the limit space of blowing-up argument. With the aid of the Pohozaev type identity, we then show that in dimension $n=5$, each isolated blowup point is in fact an isolated simple blowup point. After that, proof of Theorem \ref{thm27} is standard, except that more is involved for the blowing up limit in ruling out the bubble accumulations, see Proposition \ref{propdistanceofsingularities}.

 \begin{Remark}
In Marques \cite{Marques} and Li-Zhang \cite{Li-Zhang}, by using a classification theorem by Chen-Lin \cite{Chen-Lin} of solutions to the linearized equation of the constant scalar curvature equation on $\mathbb{R}^n$ vanishing at infinity, better estimates are obtained for error terms in the Pohozaev type identity. If such a classification theorem still holds for linearized equation of the constant $Q$-curvature equation on $\mathbb{R}^n$, then argument in \cite{Marques} still works and Proposition \ref{propinequality} still holds for $n=6$ and $n=7$, see in Remark \ref{remark6}. We should remark that in these two dimensions, estimates on the Weyl tensor at the blowup points are not necessary. Once Proposition \ref{propinequality} holds, Proposition \ref{propisolatedsingularpoints} and Proposition \ref{propdistanceofsingularities} hold for $n=6$ and $n=7$ automatically. That leads to the compactness result Theorem \ref{thm27} for $n=6$ and $n=7$. But we are not able to show such a classification so far.
\end{Remark}

To end the introduction, we introduce definition of isolated blowup points and isolated simple blowup points.

\begin{defn}
 Let $g_j$ be a sequence of Riemannian metric on a domain $\Omega \subseteq M$. Let $\{u_j\}_j$ be a sequence of positive solutions to $(\ref{equation1})$ under the background metric $g_j$ in $\Omega$. We call a point $\bar{x}\in \Omega$ an {\it isolated blowup point} of $\{u_j\}$ if there exist $\bar{C}>0$, $0< \delta < dist_{g_j}(\bar{x}, \partial \Omega)$ and $x_j \to \bar{x}$ as a local maximum of $u_j$ with $u_j(x_j)\to \infty$ satisfying
\begin{align}
&B^{g_j}_{\delta}(\bar{x}) \subseteq \Omega,\\
&\label{boundground}u_j(x)\leq \bar{C} d_{g_j}(x, x_j)^{\frac{4-n}{2}},\,\,\,\,\text{for}\,\,d_{g_j}(x, x_j)\leq \delta,
\end{align}
where $B^{g_j}_{\delta}$ is the $\delta$-geodesic ball with respect to the metric $g_j$, and $d_{g_j}(x, x_j)$ is the geodesic distance between $x$ and $x_j$ with respect to the metric $g_j$.
\end{defn}
For an isolated blowup point $x_j\to \bar{x}$ of $u_j$, we define
\begin{align*}
\bar{u}_j(r)=\frac{1}{|\partial B^{g_j}_r(x_j)|}\int_{\partial B^{g_j}_r(x_j)}u_j ds_{g_j},\,\,0<r<\delta,
\end{align*}
and
\begin{align*}
\hat{u}_j(r)=r^{\frac{n-4}{2}}\bar{u}_j(r),\,\,0<r<\delta,
\end{align*}
with $B^{g_j}_r(x_j)$ that $r$-geodesic ball centered at $x_j$, $ds_{g_j}$ the area element and $|\partial B^{g_j}_r(x_j)|$ volume of $B^{g_j}_r(x_j)$.
\begin{defn}\label{def02}
We call $x_j\to \bar{x}$ an isolated simple blowup point if it is an isolated blowup point and there exists $0<\delta_1<\delta$ independent of $j$ such that $\hat{u}_j$ has precisely one critical point in $(0, \delta_1)$, for $j$ large.
\end{defn}

{\bf Acknowledgements.} The author would like to thank Doctor Jingang Xiong and Professor Lei Zhang for helpful discussion when he read \cite{Li-Zhang} for other purpose. The author is grateful to Professor Chiun-Chuan Chen for helpful discussion for understanding \cite{Chen-Lin}.

\section{The Green's representation}

In this section, we assume that $(M^n, g)$ is a closed Riemannian manifold of dimension $n\geq 5$ with $R_g\geq 0$, and also $Q_g\geq 0$ with $Q_g(p)>0$ for some point $p\in M$.
\begin{thm}\label{thm1}
(Gursky-Malchiodi, \cite{Gursky-Malchiodi}) For a closed Riemannian manifold $(M^n, g)$ of dimension $n\geq 5$, if $R_g\geq 0$, $Q_g\geq 0$ on $M$ and also $Q_g(p)>0$ for some point $p\in M$, then the following holds:
 \begin{itemize}
\item The scalar curvature $R_g>0$ in $M$;
\item  the Paneitz operator $P_g$ is in fact positive and the Green's function $G$ of $P_g$ is positive where $G: M\times M - \{(q, q), q\in M\}\,\to\,\mathbb{R}$. Also, if $u\in C^4(M)$ and $P_g u\geq 0$ on $M$, then either $u\equiv 0$ or $u>0$ on $M$;
\item for any metric $g_1$ in the conformal class of $g$, if $Q_{g_1}\geq 0$, then $R_{g_1}>0$;
\item for any distinct points $q_1, q_2 \in M$,
\begin{align}\label{expansion1}
G(q_1, q_2)=G(q_1,q_2)=c_nd_{g}(q_1,q_2)^{4-n}(1+f(q_1, q_2)),
\end{align}
with $c_n=\frac{1}{(n-2)(n-4)\omega_{n-1}}$, $\omega_{n-1}=|S^{n-1}|$, and $d_g(q_1, q_2)$ distance between $q_1$ and $q_2$. Here $f$ is bounded and $f\to 0$ as $d_g(q_1, q_2)\to 0$ and
\begin{align}\label{ineqGrp}
|\nabla^jf|\leq C_jd_g(q_1, q_2)^{1-j}
\end{align}
for $1 \leq j \leq 4$,
\item (positive mass theorem) when the dimension $n= 5,\,6, $ or $7$, for any point $q_1\in M$, let $x=(x^1, ..., x^n)$ be the conformal normal coordinates ( see \cite{Lee-Parker}) centered at $q_1$ and $h$ be the corresponding conformal metric. For $q_2$ close to $q_1$ the Green's function $G_h(q_2, q_1)$ of the Paneitz operator $P_h$ has the expansion
    \begin{align*}
    G_h(q_2, q_1)=c_nd_{h}(q_2,q_1)^{4-n}+\alpha + f(q_2)
    \end{align*}
    with a constant $\alpha \geq 0$ and $f$ satisfying $(\ref{ineqGrp})$ and $f(q_2)\to 0$ as $q_2\to q_1$; moreover, $\alpha=0$ if and only if $(M^n, g)$ is conformally equivalent to the round sphere.
\end{itemize}

\end{thm}


Let $u\in C^{4, \alpha}(M)$ be a solution to the equation
\begin{align*}
P_gu=f\geq 0.
\end{align*}
Then we have the Green's representation
\begin{align*}
u(x)=\int_M G(x, y)f(y)dV_g(y),
\end{align*}
for $x\in M$.

Now let $u>0$ be a solution to the constant $Q$-curvature equation $(\ref{equation1})$.
Using the Green's representation,
\begin{align*}
u(x)=\frac{n-4}{2}\bar{Q} \int_M G(x, y)\,u^{\frac{n+4}{n-4}}(y)\,d V_g(y),
\end{align*}
we first show some basic estimates of the solution $u$.

\begin{lem}\label{lembounda}
For a closed Riemannian manifold $(M^n, g)$ of dimension $n\geq 5$ with $R_g>0$, $Q_g\geq 0$ on $M$ and $Q_g(p)>0$ for some point $p\in M$. Then there exists $C_1,\,C_2>0$ depending on $(M, g)$, so that for any solution $u$ to $(\ref{equation1})$, we have that
\begin{align*}
\inf_M u \leq C_1,\,\,\sup_M u \geq C_2.
\end{align*}
\end{lem}
\begin{proof}
Let $u(q)=\inf_M u$. Then by Green's representation,
\begin{align*}
u(q)&=\frac{(n-4)}{2}\bar{Q}\int_M\,G(q,y)\,u(y)^{\frac{n+4}{n-4}}\,dV_g(y)\\
&\geq u(q)^{\frac{n+4}{n-4}}*\frac{(n-4)}{2}\bar{Q}\int_M\,G(q,y)\,dV_g(y)\\
&\geq C_1^{-\frac{8}{n-4}} u(q)^{\frac{n+4}{n-4}}
\end{align*}
with $C_1$ independent of the solution $u$ and $q$, and the last inequality follows from $(\ref{expansion1})$. Therefore, the upper bound of $\inf_M u$ is established. Similar argument leads to lower bound of $\sup_M u$.

\end{proof}

Next we give an integral type inequality, which shows that if $u$ is bounded from above, then we get lower bound of $u$.
\begin{lem}\label{lemlowerbound1}
For a closed Riemannian manifold $(M^n, g)$ with dimension $n\geq 5$, $R_g>0$, and also $Q_g\geq 0$ with $Q_g(p)>0$ for some point $p\in M$. Then we have the inequality
\begin{align*}
\inf_M u \geq C (\int_M G(z, y)^p\,u(y)^{\frac{8}{n-4}\alpha p}\, d V_g(y))^{-\frac{q}{p}}
\end{align*}
where $p=\frac{n+4}{n-4} - a$, $\frac{1}{p}+\frac{1}{q}=1$, and $\alpha=\frac{(n-4)a}{8p}$, for any fixed number $\frac{4}{n-4}< a <\frac{8}{n-4}$, and $z$ is the maximum point of $u$. $C=C(a, g)>0$ is a constant.  In particular, uniform upper bound of $u$ implies uniform lower bound of $u$.
\end{lem}
\begin{proof}
Let $u(x)=\inf_M u$ and $u(z)=\sup_M u$.

By the expansion formula $(\ref{expansion1})$, there exist two constants $C_3, C_4>0$ so that
\begin{align}\label{polebound}
0< C_3 < \frac{1}{C_4}d_g(z_1,z_2)^{4-n} \leq G(z_1, z_2) \leq C_4 d_g(z_1, z_2)^{4-n},
\end{align}
for any two distinct points $z_1, z_2 \in M$.

By Green's representation at the maximum point $z$,
\begin{align*}
u(z)&= \frac{(n-4)}{2}\bar{Q}\int_M G(z, y)\,u(y)^{\frac{n+4}{n-4}}\, d V_g(y)\\
&\leq \frac{(n-4)}{2}\bar{Q}u(z)\int_M G(z, y)\,u(y)^{\frac{8}{n-4}}\, d V_g(y)
\end{align*}
so that
\begin{align*}
1&\leq \frac{(n-4)}{2}\bar{Q}\int_M G(z, y)\,u(y)^{\frac{8}{n-4}(\alpha+(1-\alpha))}\, d V_g(y)\\
&\leq \frac{(n-4)}{2}\bar{Q}(\int_M G(z, y)^p\,u(y)^{\frac{8}{n-4}\alpha p}\, d V_g(y))^{\frac{1}{p}}\,(\int_Mu(y)^{\frac{8}{(n-4)}(1-\alpha)q}\,dv_g(y))^{\frac{1}{q}}\\
&= \frac{(n-4)}{2}\bar{Q}(\int_M G(z, y)^p\,u(y)^{\frac{8}{n-4}\alpha p}\, d V_g(y))^{\frac{1}{p}}\,(\int_Mu(y)^{\frac{n+4}{n-4}}\,dv_g(y))^{\frac{1}{q}},
\end{align*}
with $\alpha$, $p$, $q$ chosen as in the lemma. Here the second inequality is by H$\ddot{\text{o}}$lder's inequality. The range of $a$ in the lemma keeps $0<\alpha<1$, $p>1$ and $q>1$, and also $p(4-n)>-n$ so that $G^p$ is integrable.

Therefore, combining with $(\ref{polebound})$ we have
\begin{align*}
\inf_M u &= u(x)=\frac{n-4}{2}\bar{Q}\int_M G(x, y) u(y)^{\frac{n+4}{n-4}}\,d V_g(y)\\
&\geq C' \int_M u(y)^{\frac{n+4}{n-4}}\,d V_g(y)\\
&\geq C (\int_M G(z, y)^p\,u(y)^{\frac{8}{n-4}\alpha p}\, d V_g(y))^{-\frac{q}{p}},
\end{align*}
where $C', C>0$ are uniform constants independent of $u$, $z$ and $x$.

\end{proof}

\section{Locally conformally flat manifolds}
In Qing-Raske \cite{Qing-Raske}, for locally conformally flat manifolds, upper bound for positive solutions to $(\ref{equation1})$ is given:
\begin{thm}\label{thmqr}
(Theorem 1.3 in \cite{Qing-Raske}) Let $(M^n g)$ be a closed locally conformally flat manifold of dimension $n \geq 5$ with positive Yamabe constant. Assume $(M, g)$ is not conformally equivalent to the round sphere. Then there exists $C>0$ so that for any positive function $u$ if the metric $g_1=u^{\frac{4}{n-4}}g$ is of positive scalar curvature and constant $Q$-curvature $1$, then $u\leq C$.
\end{thm}
For estimate of lower bound of $u$, they need assumption on the so called Poincar$\acute{\text{e}}$ exponent. Now for our problem, since $R_g>0$, the Yamabe constant is positive. The above theorem applies. Combining with Lemma \ref{lemlowerbound1}, we obtain that
\begin{thm}\label{thmconformally_flat}
 Let $(M^n, g)$ be a closed locally conformally flat Riemannian manifold of dimension $n\geq 5$. Assume that $(M, g)$ is not conformally equivalent to the round sphere. If $R_g\geq 0$, and also $Q_g\geq 0$ with $Q_g(p)>0$ for some point $p\in M$, then there exists $C>0$ and $C'=C'(\alpha)$ for any $0<\alpha<1$, so that for any solution $u$ of $(\ref{equation1})$,
\begin{align}
&\label{boundflatness} \frac{1}{C}< u < C,\\
&\label{boundflatness1} |u|_{C^{4, \alpha}}\leq C'.
\end{align}
\end{thm}
\begin{proof}
For $Q_{u^{\frac{4}{n-4}}g}=\bar{Q}$, by Theorem \ref{thm1} the scalar curvature $R_{u^{\frac{4}{n-4}}g}>0$. The estimate $(\ref{boundflatness})$ follows from Theorem \ref{thmqr} and Lemma \ref{lemlowerbound1}. To establish $(\ref{boundflatness1})$, one can either use ellipticity of the equation $(\ref{equation1})$ with a bootstrapping argument or take derivatives on the Green's representation and use $(\ref{expansion1})$. This completes the proof of Theorem \ref{thmconformally_flat}.
\end{proof}

\section{A maximum principle}

In this section we give a maximum principle for smooth domains with boundary in the manifold $(M, g)$ defined in Lemma \ref{lembounda}, which is a modification of the maximum principle given by Gursky and Malchiodi, see Lemma \ref{lemmaximumprinciple}. As an application, we give a lower bound estimate of the blowing up sequence.
\begin{lem}\label{lemboundScurvature}
Let $(\bar{\Omega}, g)$ be a compact Riemannian manifold with boundary $\partial \Omega$ of dimension $n\geq 5$. Let $\Omega$ be the interior of $\bar{\Omega}$. Assume the scalar curvature $R_g\geq 0$ in $\bar{\Omega}$ and $R_g>0$ at points on the boundary, and also $Q_g\geq 0$ in $\bar{\Omega}$. Then $R_g>0$ in $\bar{\Omega}$.
\end{lem}
\begin{proof}
The proof is the same as that for closed manifolds. The $Q$- curvature is expressed as
\begin{align*}
Q_g=-\frac{1}{2(n-1)}\Delta_gR_g+c_1(n)R_g^2-c_2(n)|Ric|_g^2
\end{align*}
with $c_1(n), c_2(n)$ positive. By the non-negativity of $Q_g$,
\begin{align*}
\frac{1}{2(n-1)}\Delta_gR_g\leq c_1(n)R_g^2.
\end{align*}
By strong maximum principle and the boundary condition, $R_g>0$ in $\bar{\Omega}$.
\end{proof}

\begin{lem}\label{lemmaximumprinciple}
Let $(M^n, g)$ be a closed Riemannian manifold of dimension $n\geq 5$ with $R_g\geq 0$, and $Q_g\geq 0$. Let $\Omega \subseteq M$ be an open domain with smooth boundary $\partial \Omega$ so that $\bar{\Omega}=\Omega \bigcup \partial \Omega$. Assume that $u\in C^4(\bar{\Omega})$ satisfies that
\begin{align}
P_gu\geq 0\,\,\text{in}\,\Omega,
\end{align}
with $u>0$ on $\partial \Omega$. Let $\tilde{g} =u^{\frac{4}{n-4}}g$ be the conformal metric in a neighborhood $\mathcal {U}$ of $\partial \Omega$ where $u>0$. If the scalar curvature of $(\mathcal {U}, \tilde{g})$ satisfies $R_{\tilde{g}}(p)>0$ for all points $p \in \partial\Omega$, then $u>0$ in $\Omega$.
\end{lem}
\begin{proof}
Our conditions on the boundary guarantee that all the argument is focused on the interior and then the argument is the same as in the proof of the maximum principle by Gursky and Malchiodi. For completeness, we present the proof.

We define the function
\begin{align*}
u_{\lambda}=(1-\lambda)+\lambda u
\end{align*}
for $\lambda\in [0, 1]$, so that $u_0=1$ and $u_1=u$. We assume that \begin{align*}
\min_{\overline{\Omega}} u\leq 0.
\end{align*}
Then there exists $\lambda_0\in (0, 1]$ so that
\begin{align*}
\lambda_0=\inf\{\lambda \in(0, 1],\,\inf_{\overline{\Omega}} u_{\lambda}=0\}.
\end{align*}
By definition, for $0<\lambda<\lambda_0$, $u_{\lambda}>0$. For the metric
\begin{align*}
g_{\lambda}=u_{\lambda}^{\frac{4}{n-4}}g,
\end{align*}
the $Q$-curvature satisfies
\begin{align*}
Q_{g_{\lambda}}\geq 0\,\,\text{in}\,\Omega,
\end{align*}
for $0<\lambda<\lambda_0$. That follows from the conformal transformation formula
\begin{align*}
Q_{g_{\lambda}}&=\frac{2}{n-4}u_{\lambda}^{-\frac{(n+4)}{(n-4)}}P_gu_{\lambda}\\
&=\frac{2}{n-4}u_{\lambda}^{-\frac{(n+4)}{(n-4)}}((1-\lambda)P_g(1)+\lambda P_g u)\\
&=\frac{2}{n-4}u_{\lambda}^{-\frac{(n+4)}{(n-4)}}((1-\lambda)\frac{(n-4)}{2}Q_g+\lambda P_g u)\\
&\geq (1-\lambda)Q_gu_{\lambda}^{-\frac{n+4}{n-4}}\geq 0.
\end{align*}

Under the conformal transformation, the scalar curvature of $g_{\lambda}$ satisfies
\begin{align*}
R_{g_{\lambda}}&=u_{\lambda}^{-\frac{n}{n-4}}\big(-\frac{4(n-1)}{(n-4)}\Delta_g u_{\lambda}- \frac{8(n-1)}{(n-4)^2}\frac{|\nabla_gu_{\lambda}|^2}{u_{\lambda}}+R_gu_{\lambda}\big)\\
&=u_{\lambda}^{-\frac{n}{n-4}}\big(-\frac{4(n-1)}{(n-4)}\lambda\Delta_g u- \frac{8(n-1)}{(n-4)^2}\frac{\lambda^2|\nabla_gu|^2}{(1-\lambda)+\lambda u}+R_gu_{\lambda}\big)\\
&\geq u_{\lambda}^{-\frac{n}{n-4}}\big(-\frac{4(n-1)}{(n-4)}\lambda\Delta_g u- \frac{8(n-1)}{(n-4)^2}\frac{\lambda|\nabla_gu|^2}{ u}+ \lambda R_gu\big)\\
&=\lambda\big(\frac{u}{u_{\lambda}}\big)^{\frac{n}{n-4}}R_{\tilde{g}}>0
\end{align*}
on $\partial \Omega$ for $0<\lambda< \lambda_0$.
Then by Lemma \ref{lemboundScurvature},
\begin{align*}
R_{g_{\lambda}}>0 \,\,\text{in}\,\Omega
\end{align*}
for $0< \lambda < \lambda_0$. Again by the conformal transformation formula of scalar curvature,
\begin{align*}
\Delta_g u_{\lambda}\leq \frac{(n-4)}{4(n-1)}R_gu_{\lambda}\,\,\,\text{in}\,\,\Omega.
\end{align*}
By taking limit $\lambda\nearrow \lambda_0$, this also holds at $\lambda=\lambda_0$. But
\begin{align*}
u_{\lambda}=(1-\lambda)+\lambda u>0
\end{align*}
on $\partial \Omega$ for $0\leq\lambda\leq 1$. By strong maximum principle, $u_{\lambda_0}>0$ in $\bar{\Omega}$, contradicting with choice of $\lambda_0$. Therefore, for all $0\leq \lambda \leq 1$, \begin{align*}
u_{\lambda}>0\,\,\, \text{in}\,\, \Omega.
\end{align*}
In particular, $u>0$ in $\Omega$.

\end{proof}

\vskip0.3cm
\begin{thm}\label{thmlowerbound}
Let $(M^n, g)$ be a closed Riemannian manifold of dimension $n\geq 5$ with $R_g\geq 0$, and also $Q_g\geq 0$ with $Q_g(p_0)>0$ for some point $p_0\in M$. There exists $C>0$ so that if there exists a sequence of positive solutions $\{u_j\}_{j=1}^{\infty}$ of $(\ref{equation1})$ so that
\begin{align*}
M_j=u_j(x_j)=\sup_{M}u_j \to \infty
\end{align*}
as $j\to \infty$, then
\begin{align}\label{ineqnblowuplowerbound}
u_j(p)\geq C M_j^{-1}d_g^{4-n}(p, x_j)
\end{align}
for any $p\in M$ such that $d_g(p, x_j)\geq M_j^{-\frac{2}{n-4}}$.
\end{thm}
\begin{proof}
To prove the theorem, we only need to show that there exists $C>0$ so that for any blowing up sequence, there exists a subsequence so that $(\ref{ineqnblowuplowerbound})$ holds.

Let $x=(x^1,...,x^n)$ be normal coordinates in a small geodesic ball centered at $x_j$ with radius $\delta>0$ and $x_j$ the origin. Let $y=M_j^{\frac{2}{n-4}}x$ and the metric $h_j$ be given by $(h_j)_{pq}(y)=g_{pq}(M_j^{-\frac{2}{n-4}} y)$. Let
\begin{align*}
v_j(y)=M_j^{-1}u_j(\exp_{x_j}(M_j^{-\frac{2}{n-4}}y))\,\,\,\text{for}\,\,|y|\leq \delta M_j^{\frac{2}{n-4}}.
\end{align*}
 Then
\begin{align*}
&0<\, v_j(y)\,\leq v_j(0)=1,\\
&P_{h_j}v_j(y)=\frac{(n-4)}{2}\bar{Q}v_j(y)^{\frac{n+4}{n-4}}\,\,\,\,\text{for}\,\,|y|\leq\delta M_j^{\frac{2}{n-4}}.
\end{align*}
Here $h_j$ converges to Euclidean metric on $\mathbb{R}^n$ in $C^k$ norm for any $k\geq 0$. By ellipticity, we have, after passing to a subsequence( still denoted as $\{v_j\}$), $v_j\to v$ in $C_{loc}^4(\mathbb{R}^n)$ and $v$ satisfies
\begin{align*}
&0\leq v(y) \leq v(0)=1\,\,\,\,\text{in}\,\,\mathbb{R}^n,\\
&\Delta^2v(y)=\frac{(n-4)}{2}\bar{Q}v(y)^{\frac{n+4}{n-4}}\,\,\,\,\text{in}\,\,\mathbb{R}^n.
\end{align*}
Also, since $R_{h_j}>0$ and $R_{u_j^{\frac{4}{n-4}}g}>0$ on $M$, by conformal transformation formula of scalar curvature,
\begin{align*}
\Delta_{h_j}v_j\leq \frac{(n-4)}{4(n-1)}R_{h_j}v_j.
\end{align*}
Passing to the limit we have
\begin{align*}
\Delta v(y)\leq 0\,\,\,\,\text{in}\,\,\mathbb{R}^n.
\end{align*}
By strong maximum principle, since $v(0)=1$, we have that $v(y)>0$ in $\mathbb{R}^n$. Then by the classification theorem of C.S. Lin(\cite{Lin}), we have that
\begin{align*}
v(y)= \big(\frac{1}{1+4^{-1}|y|^2}\big)^{\frac{n-4}{2}}\,\,\,\,\text{in}\,\,\mathbb{R}^n.
\end{align*}
We will abuse the notation $v(|y|)=v(y)$. Therefore, for fixed $R>0$, for $j$ large,
\begin{align*}
\frac{1}{2}\big(\frac{1}{1+4^{-1}R^2}\big)^{\frac{n-4}{2}}M_j\leq u_j(\exp_{x_j}(x))\leq M_j\,\,\,\,\text{for}\,\,|x|\leq R M_j^{-\frac{2}{n-4}}.
\end{align*}
For any $\epsilon>0$, there exists $j_0>0$ so that for $j>j_0$,
\begin{align*}
\|v_j-v\|_{C^4}\leq \epsilon\,\,\,\,\text{for}\,\,|y|\leq 2.
\end{align*}
We define $\phi_j: M-\{x_j\}\to \mathbb{R}$ as
\begin{align*}
\phi_j(p)=u_j(p)-\tau M_j^{-1}G_{x_j}(p)
\end{align*}
with $G_{x_j}(p)=G(x_j, p)$ Green's function of Paneitz operator and $\tau>0$ a small constant to be chosen. We will use maximum principle to show that for $\epsilon, \tau>0$ small,
\begin{align*}
\phi_j>0\,\,\,\,\text{in} \,\,M-B_{M_j^{-\frac{2}{n-4}}}(x_j),\,\,\,\text{for}\,\,j>j_0.
\end{align*}
Here $B_{M_j^{-\frac{2}{n-4}}}(x_j)$ denote the geodesic $M_j^{-\frac{2}{n-4}}$-ball centered at $x_j$ in $(M, g)$. If this holds, we will choose $\{u_j\}_{j>j_0}$ as the subsequence and the theorem is proved.

It is clear that
\begin{align*}
P_g\phi_j=P_gu_j=\frac{n-4}{2}\bar{Q}u_j^{\frac{n+4}{n-4}}>0\,\,\,\text{in} \,\,M-B_{M_j^{-\frac{2}{n-4}}}(x_j).
\end{align*}
To apply the maximum principle, we only need to verify sign of $\phi_j$ and related scalar curvature on $\partial B_{M_j^{-\frac{2}{n-4}}}(x_j)$.

First, for $|x|= M_j^{-\frac{2}{n-4}}$, we choose $\epsilon$ small so that for $j>j_0$
\begin{align*}
u_j(\exp_{x_j}(x))=M_jv_j(M_j^{\frac{2}{n-4}}x)\geq \frac{1}{2}v(1)M_j;
\end{align*}
while by $(\ref{polebound})$,
\begin{align*}
M_j^{-1}G_{x_j}(\exp_{x_j}(x))\leq C_4 M_j.
\end{align*}
We take $\tau< \frac{v(1)}{4C_4}$. Then
\begin{align*}
\phi_j>0\,\,\,\text{on} \,\,\partial B_{M_j^{-\frac{2}{n-4}}}(x_j),\,\,\,\text{for}\,\,j>j_0.
\end{align*}
Now let $\tilde{g}_j=\phi_j^{\frac{4}{n-4}}g_j$ in small neighborhood of $\partial B_{M_j^{-\frac{2}{n-4}}}(x_j)$ where $\phi_j>0$. By conformal transformation,
\begin{align*}
R_{\tilde{g}_j}&=\phi_j^{-\frac{n}{n-4}}\big(-\frac{4(n-1)}{(n-4)}\Delta_g \phi_j- \frac{8(n-1)}{(n-4)^2}\frac{|\nabla_g\phi_j|^2}{\phi_j}+R_g\phi_j\big).
\end{align*}
Note that $R_g\phi_j>0$ on $\partial B_{M_j^{-\frac{2}{n-4}}}(x_j)$. We only need to show that
\begin{align}\label{positiveScurvature}
-\frac{4(n-1)}{(n-4)}(\Delta_g \phi_j+ \frac{2}{(n-4)}\frac{|\nabla_g\phi_j|^2}{\phi_j})>0 \,\,\,\text{on} \,\,\partial B_{M_j^{-\frac{2}{n-4}}}(x_j),\,\,\,\text{for}\,\,j>j_0.
\end{align}
Recall that
\begin{align*}
(\Delta_g u_j+ \frac{2}{(n-4)}\frac{|\nabla_gu_j|^2}{u_j})=M_j^{1+\frac{4}{n-4}}(\Delta_{h_j}v_j+ \frac{2}{(n-4)}\frac{|\nabla_{h_j}v_j|^2}{v_j}).
\end{align*}
Also,
\begin{align*}
(\Delta_{h_j}v_j+ \frac{2}{(n-4)}\frac{|\nabla_{h_j}v_j|^2}{v_j})&\to (\Delta v+ \frac{2}{(n-4)}\frac{|\nabla v |^2}{v})\\
&=2(4-n)(|y|^2+4)^{-\frac{n}{2}}(|y|^2+2n)\,\,+\frac{2}{n-4}\frac{(4-n)^2(|y|^2+4)^{2-n}|y|^2}{(|y|^2+4)^{\frac{4-n}{2}}}\\
&=2(4-n)(|y|^2+4)^{-\frac{n}{2}}(|y|^2+2n)\,\,+\,2(n-4)(|y|^2+4)^{-\frac{n}{2}}|y|^2\\
&=4n(4-n)(|y|^2+4)^{-\frac{n}{2}}<0\,\,\,\,\text{at}\,\,|y|=1.
\end{align*}
Then we can choose $\epsilon< \frac{1}{100^n}|v|_{C^4(B_1(0))}$.
Combining with the fact that
\begin{align*}
|D_g^kG_{p}(q)|\leq C_k d_g^{4-n-k}(p, q)\,\,\,\,\text{for}\,\,0\leq k\leq 4,
\end{align*}
for any distinct points $p, q \in M$ with constant $C_k>0$ independent of $p, q$, there exists $\tau>0$ only depending on $C_k$ and $\epsilon$ so that
\begin{align*}
&\tau M_j^{-1}|\Delta_g G_{x_j}(\exp_{x_j}(M_j^{-\frac{2}{n-4}}y))|\,<\,- M_j^{1+\frac{4}{n-4}}\frac{\Delta v}{4(2n+1)},\,\,\,\,\text{and}\\
&\frac{|\nabla_g\phi_j|^2}{\phi_j}\leq \frac{5}{4}M_j^{1+\frac{4}{n-4}}\frac{|\nabla v |^2}{v}\,\,\,\text{at}\,\,|y|=1,\,\,\,\,\text{for}\,\,j>j_0.
\end{align*}
 Therefore, for $j>j_0$, $(\ref{positiveScurvature})$ holds,  which implies that
 \begin{align*}
 R_{\tilde{g}_j}>0\,\,\text{on}\,\,\partial B_{M_j^{-\frac{2}{n-4}}}(x_j).
 \end{align*}
 By Lemma \ref{lemmaximumprinciple}, $\phi_j> 0$ in $M - B_{M_j^{-\frac{2}{n-4}}}(x_j)$. Recall that $\epsilon$ and $\tau$ are chosen independent of choice of the sequence. This completes the proof of the theorem.
\end{proof}

\section{A Pohozaev type identity}
In this section we introduce a Pohozaev type identity related to the constant $Q$-curvature equation. It will provide local information of the solutions in later use.

Let $(M^n, g)$ be a closed Riemannian manifold of dimension $n\geq 5$ with $R_g\geq 0$, and also $Q_g\geq 0$ with $Q_g(p_0)>0$ for some point $p_0\in M$. Let $u$ be a positive solutions to $(\ref{equation1})$. For any geodesic ball $\Omega=B_{\delta}(q)$ in $M$ with $2\delta$ less than injectivity radius of $(M, g)$, we let $x=(x^1,...,x^n)$ be geodesic normal coordinates centered at $q$ so that $g_{ij}(0)=\delta_{ij}$ and the Christoffel symbols $\Gamma_{ij}^k(0)=0$. In this section, the gradient $\nabla$, Laplacian $\Delta$, divergent $\text{div}$, volume element $dx$, area element $ds$, $\sigma$-ball $B_{\sigma}$ and $|x|^2=(x^1)^2+..+(x^n)^2$ are all with respect to the Euclidean metric. Define
\begin{align*}
&\mathcal {P}(u)=\int_{\Omega}(x\cdot \nabla u+ \frac{n-4}{2}u)\Delta^2u dx\\
&=\int_{\Omega} [\frac{n-4}{2}\text{div}(u\nabla(\Delta u)-\Delta u\nabla u)+\text{div}((x\cdot\nabla u)\nabla(\Delta u)-\nabla (x\cdot\nabla u)\Delta u + \frac{1}{2} (\Delta u)^2x)] dx\\
&=\int_{\partial \Omega} \,\frac{n-4}{2}(u\frac{\partial}{\partial \nu}(\Delta u)-\Delta u  \frac{\partial}{\partial \nu}u)+ ((x\cdot\nabla u)\frac{\partial}{\partial \nu}(\Delta u)- \frac{\partial}{\partial \nu} (x\cdot\nabla u)\Delta u + \frac{1}{2} (\Delta u)^2x\cdot \nu) ds,
\end{align*}
where $\nu$ is the outer pointing normal vector of $\partial \Omega$ in Euclidean metric. Then using $(\ref{equation1})$ we have
\begin{align*}
\mathcal {P}(u)&=\int_{\Omega}(x\cdot \nabla u+ \frac{n-4}{2}u)(\Delta^2-P_g)u\,+\,(x\cdot \nabla u+ \frac{n-4}{2}u)\,P_gu\,dx\\
&=\int_{\Omega}(x\cdot \nabla u+ \frac{n-4}{2}u)(\Delta^2-P_g)u\,+\,\frac{n-4}{2}\bar{Q}\,(x\cdot \nabla u+ \frac{n-4}{2}u) u^{\frac{n+4}{n-4}} dx\\
&=\int_{\Omega}(x\cdot \nabla u+ \frac{n-4}{2}u)(\Delta^2-P_g)u\,+\,\frac{(n-4)^2}{4n}\bar{Q}\,\text{div}(u^{\frac{2n}{n-4}}x) dx\\
&=\int_{\Omega}(x\cdot \nabla u+ \frac{n-4}{2}u)(\Delta^2-P_g)u\,dx\,+\,\frac{(n-4)^2}{4n}\bar{Q}\int_{\partial \Omega}(x\cdot \nu)u^{\frac{2n}{n-4}}\, dx.
\end{align*}
Using the expression $(\ref{operatorP})$, we have
\begin{align*}
(\Delta^2-P_g)u=(\Delta^2-\Delta_g^2)u+\text{div}_g(a_nR_g g -b_nRic_g)\nabla_gu - \frac{n-4}{2}Q_gu.
\end{align*}
Since $\Gamma_{ij}^k(0)=0$ and $g_{ij}(0)=\delta_{ij}$,
\begin{align*}
(\Delta^2-\Delta_g^2)u&=(\delta^{pq}\delta^{ij}\nabla_p\nabla_q\nabla_i\nabla_j-g^{pq}g^{ij}\nabla^g_p\nabla^g_q\nabla^g_i\nabla^g_j)u\\
&=(\delta^{pq}\delta^{ij}- g^{pq}g^{ij})\nabla_p\nabla_q\nabla_i\nabla_ju+ O(|x|)|D^3u|+O(1)|D^2u|+O(1)|D u|\\
&=O(|x|^2)|D^4u|+O(|x|)|D^3u|+O(1)|D^2u|+O(1)|D u|.
\end{align*}
It follows that there exists $C>0$ depending on $|Rm_g|_{L^{\infty}(\Omega)}$, $|Q_g|_{C(\Omega)}$ and $|Ric_g|_{C^1(\Omega)}$ such that
\begin{align}\label{ineqbounderrorterms}
|(\Delta^2-P_g)u|\leq C(|x|^2|D^4u|+\,|x|\,|D^3u|+\,|D^2u|+\,|D u|+\,u).
\end{align}

\section{Upper bound estimates near isolated simple blowup points}\label{section6}
In this section we perform a parallel approach of \cite{Li-Zhu} to show upper bound estimates of the solutions to $(\ref{equation1})$ near an isolated simple blowup point, see Proposition \ref{propupperbound}. We start with a Hanark type inequality near an isolated blowup point.
\begin{lem}\label{lemH}
Let $(M^n, g)$ be a closed Riemannian manifold of dimension $5\leq n \leq 9$ with $R_g\geq 0$, and also $Q_g\geq 0$ with $Q_g(p_0)>0$ for some point $p_0\in M$. Let $\{u_j\}$ be a sequence of positive solutions to $(\ref{equation1})$ and $x_j\to \bar{x}$ be an isolated blowup point. Then there exists a constant $C>0$ such that for any $0< r< \frac{\delta}{3}$, we have
\begin{align}\label{ineqH}
\max_{q\in B_{2r}(x_j)- B_{\frac{r}{2}}(x_j)}u_j(q) \leq C \min_{q\in B_{2r}(x_j)- B_{\frac{r}{2}}(x_j)}u_j(q).
\end{align}
\end{lem}
\begin{proof}
Let $x=(x^1,...,x^n)$ be geodesic normal coordinates centered at $x_j$. Here $\delta>0$ can be chosen small so that the coordinates exist. Let $y=r^{-1}x$. Define
\begin{align*}
v_j(y)=r^{\frac{n-4}{2}}u_j(\exp_{x_j}(ry))\,\,\,\text{for}\,\,|y| < 3.
\end{align*}
Then
\begin{align*}
&v_j(y)\leq \bar{C}|y|^{-\frac{n-4}{2}}\,\,\,\text{for}\,\,|y|<3,\\
&v_j(y)\leq 3^{\frac{n-4}{2}}\bar{C}\,\,\,\,\,\,\,\,\text{for}\,\,\frac{1}{3}<|y|<3.
\end{align*}
We denote
\begin{align*}
\Omega_r=B_{3r}(x_j)- B_{\frac{r}{3}}(x_j).
\end{align*}
By Green's representation,
\begin{align*}
v_j(y)&=r^{\frac{n-4}{2}} u_j(exp_{x_j}(ry))=r^{\frac{n-4}{2}}\int_M G(exp_{x_j}(ry), q)u_j(q)^{\frac{n+4}{n-4}}d V_g(q)\\
&=r^{\frac{n-4}{2}}\int_{\Omega_r}G(exp_{x_j}(ry), q)u_j(q)^{\frac{n+4}{n-4}}d V_g(q)+\,r^{\frac{n-4}{2}}\,\int_{M-\Omega_r} G(exp_{x_j}(ry), q)u_j(q)^{\frac{n+4}{n-4}}d V_g(q)
\end{align*}
We {\bf claim} that for $\frac{5}{12}\leq |y|\leq \frac{12}{5}$, if
\begin{align}\label{halfbounds}
v_j(y)\geq 2 r^{\frac{n-4}{2}}\int_{\Omega_r}G(exp_{x_j}(ry), q)u_j(q)^{\frac{n+4}{n-4}}d V_g(q),
\end{align}
then there exists $C>0$ independent of $j$, $x_j$, $r$ and $y$, such that for any $\frac{5}{12}\leq |z|\leq \frac{12}{5}$,
\begin{align}\label{inqH1}
v_j(z)\geq C v_j(y).
\end{align}
In fact, by $(\ref{polebound})$, there exists $C>0$, such that
\begin{align*}
G(\exp_{x_j}(ry), q)\leq C G(\exp_{x_j}(rz), q)
\end{align*}
for $q\in M - \Omega_r$. Therefore,
\begin{align*}
\frac{1}{2}v_j(y)&\leq r^{\frac{n-4}{2}}\int_{M-\Omega_r}G(exp_{x_j}(ry), q)u_j(q)^{\frac{n+4}{n-4}}d V_g(q)\\
&\leq C r^{\frac{n-4}{2}} \int_{M-\Omega_r}G(exp_{x_j}(rz), q)u_j(q)^{\frac{n+4}{n-4}}d V_g(q)\\
&\leq C v_j(z).
\end{align*}
This proves the {\bf claim}.\vskip0.2cm

We denote
\begin{align*}
\mathcal {C}=\{y\in \mathbb{R}^n,\,\frac{5}{12}\leq |y|\leq \frac{12}{5},\,\,\text{so that}\,\,(\ref{halfbounds})\,\,\text{fails for}\,\,y\}.
\end{align*}
 We choose $\frac{5}{12}\leq |y|\leq \frac{12}{5}$ with
 \begin{align*}
 v_j(y)\geq\frac{1}{2} \sup_{\frac{5}{12}\leq|z|\leq \frac{12}{5}}v_j(z).
 \end{align*}
If $y\notin \mathcal {C}$, then using the claim, we are done. If $y\in \mathcal {C}$, we will prove that the Harnack inequality $(\ref{ineqH})$ still holds.

By H$\ddot{\text{o}}$lder's inequality,
\begin{align*}
u_j(\exp_{x_j}(ry))&\leq 2 \int_{\Omega_r}G(exp_{x_j}(ry), q)u_j(q)^{\frac{n+4}{n-4}}d V_g(q)\\
&\leq 2(\int_{\Omega_r}G(exp_{x_j}(ry),q)^{\alpha}d V_g(q))^{\frac{1}{\alpha}}(\int_{\Omega_r}u_j(q)^{\frac{n+4}{n-4}\beta}d V_g(q))^{\frac{1}{\beta}}\\
&\leq C(\alpha) r^{4-n+\frac{n}{\alpha}}(\int_{\Omega_r}u_j(q)^{\frac{n+4}{n-4}\beta}d V_g(q))^{\frac{1}{\beta}}\\
&\leq C(\alpha) r^{4-n+\frac{n}{\alpha}}\,(\bar{C}3^{\frac{n-4}{2}}r^{\frac{4-n}{2}})^{\frac{n+4}{n-4}(1-\frac{1}{\beta})}(\int_{\Omega_r}u_j(q)^{\frac{n+4}{n-4}}d V_g(q))^{\frac{1}{\beta}}\\
&\leq C(\alpha) r^{4-n+\frac{n}{\alpha}}\,(\bar{C}3^{\frac{n-4}{2}}r^{\frac{4-n}{2}})^{\frac{n+4}{n-4}(1-\frac{1}{\beta})}(\int_{\Omega_r}C_4(4r)^{n-4}G(\exp_{x_j}(rz), q)u_j(q)^{\frac{n+4}{n-4}}d V_g(q))^{\frac{1}{\beta}}\\
&\leq C(\alpha)r^{4-n+\frac{n}{\alpha}} (\bar{C}3^{\frac{n-4}{2}}r^{\frac{4-n}{2}})^{\frac{n+4}{n-4}(1-\frac{1}{\beta})}r^{\frac{n-4}{\beta}}u_j(\exp_{x_j}(rz))^{\frac{1}{\beta}}\\
&=C(\alpha, \bar{C},n)r^{(2-\frac{n}{2})(1-\frac{1}{\beta})}u_j(\exp_{x_j}(rz))^{\frac{1}{\beta}}.
\end{align*}
for any $\frac{1}{3}\leq |z|\leq 3$, where $1 < \alpha < \frac{n}{n-4}$, $\frac{1}{\alpha}+\frac{1}{\beta}=1$ so that $\beta> \frac{n}{4}$. Here we have used $(\ref{boundground})$ and $(\ref{polebound})$.

Since
\begin{align*}
\frac{n+4}{n-4}\,>\,\frac{n}{4}
\end{align*}
for $5\leq n\leq 9$, we set $\beta=\frac{n+4}{n-4}$ and obtain
\begin{align}\label{ineqHinside}
u_j(\exp_{x_j}(rz))&\geq C(\bar{C},n)\,r^4\,u_j(\exp_{x_j}(ry))^{\frac{n+4}{n-4}}\\
&\geq C(\bar{C},n)\,r^4\,(2^{-1}u_j(q))^{\frac{n+4}{n-4}},
\end{align}
for all $q\in \,B_{\frac{12r}{5}}(x_j)- B_{\frac{5r}{12}}(x_j)$ and $\frac{1}{2}\leq |z|\leq 2$, where $5\leq n\leq 9$.

For any $\frac{1}{2}\leq |z|\leq 2$,
\begin{align}\label{ineqgradients}
|\nabla_g u_j|(\exp_{x_j}(rz))&\leq \frac{n-4}{2}\bar{Q}\int_{ B_{\frac{12r}{5}}(x_j)- B_{\frac{5r}{12}}(x_j)}|\nabla_gG(\exp_{x_j}(rz), q)|\,u_j(q)^{\frac{n+4}{n-4}}dV_g(q)\\
&+\,\frac{n-4}{2}\bar{Q}\int_{M - \big( B_{\frac{12r}{5}}(x_j)- B_{\frac{5r}{12}}(x_j)\big)}|\nabla_gG(\exp_{x_j}(rz), q)|\,u_j(q)^{\frac{n+4}{n-4}}dV_g(q).
\end{align}
Note that for $\frac{1}{2}\leq |z|\leq 2$,
\begin{align}\label{ineqHoutside}
u_j(\exp_{x_j}(rz))&\geq \,\frac{n-4}{2}\bar{Q}\int_{M - \big( B_{\frac{12r}{5}}(x_j)- B_{\frac{5r}{12}}(x_j)\big)}G(\exp_{x_j}(rz), q)\,u_j(q)^{\frac{n+4}{n-4}}dV_g(q)\\
&\geq C r \int_{M - \big( B_{\frac{12r}{5}}(x_j)- B_{\frac{5r}{12}}(x_j)\big)}|\nabla_g \,G(\exp_{x_j}(rz), q)|\,u_j(q)^{\frac{n+4}{n-4}}dV_g(q),
\end{align}
for a uniform constant $C$ independent of $j$ and the choice of points, where for the last inequality we have used $(\ref{expansion1})$.

 Combining $(\ref{ineqHinside})$, $(\ref{ineqHoutside})$ and $(\ref{ineqgradients})$, for $\frac{1}{2}\leq |z|\leq 2$ we have the gradient estimate
\begin{align*}
|\nabla_g \log(u_j(\exp_{x_j}(rz)))|&=\frac{|\nabla_g u_j(\exp_{x_j}(rz))|}{u_j(\exp_{x_j}(rz))}\\
&\leq\,\frac{1}{u_j(\exp_{x_j}(rz))}\,\frac{n-4}{2}\bar{Q}\int_{ B_{\frac{12r}{5}}(x_j)- B_{\frac{5r}{12}}(x_j)}|\nabla_gG(\exp_{x_j}(rz), q)|\,u_j(q)^{\frac{n+4}{n-4}}dV_g(q)\\
&+\,\frac{1}{u_j(\exp_{x_j}(rz))}\,\frac{n-4}{2}\bar{Q}\int_{M - \big( B_{\frac{12r}{5}}(x_j)- B_{\frac{5r}{12}}(x_j)\big)}|\nabla_gG(\exp_{x_j}(rz), q)|\,u_j(q)^{\frac{n+4}{n-4}}dV_g(q)\\
&\leq\,\frac{n-4}{2}\bar{Q}\int_{ B_{\frac{12r}{5}}(x_j)- B_{\frac{5r}{12}}(x_j)}|\nabla_gG(\exp_{x_j}(rz), q)|\,C(\bar{C}, n)^{-1}r^{-4}2^{-\frac{n+4}{n-4}}dV_g(q)\\
&+\,C^{-1}r^{-1}\\
&\leq C(\bar{C}, n)(r^{3}r^{-4}+r^{-1})\\
&=\,C(\bar{C}, n)r^{-1},
\end{align*}
where $C(\bar{C}, n)$ is some uniform constant depending on $\bar{C}$, the manifold and $n$. For any two points $p,\,q\in B_{2r}(x_j)-B_{\frac{r}{2}}(x_j)$, by the gradient estimate,
\begin{align*}
\frac{u_j(p)}{u_j(q)}\leq e^{C(\bar{C}, n)r^{-1}\,d_g(p,q)}\leq e^{4nC(\bar{C}, n)}.
\end{align*}
This completes the proof of Harnack inequality.

\end{proof}

Next we show that near an isolated blowup point, after rescaling the functions $u_j$ converge to the standard solution in $\mathbb{R}^n$.
\begin{lem}\label{lemlimitmodel}
Let $(M^n, g)$ be a closed Riemannian manifold of dimension $5\leq n \leq 9$ with $R_g\geq 0$, and also $Q_g\geq 0$ with $Q_g(p_0)>0$ for some point $p_0\in M$. Let $\{u_j\}$ be a sequence of positive solutions to $(\ref{equation1})$ and $x_j\to \bar{x}$ be an isolated blowup point. Let $M_j=u_j(x_j)$. For any given $R_j\to +\infty$ and positive numbers $\epsilon_j\to 0$, after possibly passing to a subsequence $u_{k_j}$ and $x_{k_j}$( still denoted as $u_j$ and $x_j$), it holds that
\begin{align}\label{ineqmeasurelimit}
&\|M_j^{-1}u_j(\exp_{x_j}(M_j^{-\frac{2}{n-4}}y))\,-\,\big(1+4^{-1}|y|^2\big)^{-\frac{n-4}{2}}\|_{C^4(B_{2R_j})}\\
&+\|M_j^{-1}u_j(\exp_{x_j}(M_j^{-\frac{2}{n-4}}y))\,-\,\big(1+4^{-1}|y|^2\big)^{-\frac{n-4}{2}}\|_{H^4(B_{2R_j})}\,\leq \epsilon_j,
\end{align}
and
\begin{align}\label{ineqrulerlimit}
\frac{R_j}{\log(M_j)}\to 0,\,\,\text{as}\,\,j\to\,\infty.
\end{align}
\end{lem}

\begin{proof}
The proof is almost the same as in \cite{Li-Zhu}. Let $x=(x^1,...,x^n)$ be geodesic normal coordinates centered at $x_j$, $y=r^{-1}x$ and the metric $h=r^{-2}g$ be the rescaled metric so that $(h_j)_{pq}(y)=(g_j)_{pq}(ry)$ in normal coordinates. Define
\begin{align*}
v_j(y)=M_j^{-1}u_j(\exp_{x_j}(M_j^{-\frac{2}{n-4}}y))\,\,\,\text{for}\,\,|y| < \delta\,M_j^{\frac{2}{n-4}}.
\end{align*}
Then $v_j$ satisfies
\begin{align}
&P_{h_j}v_j(y)=\frac{n-4}{2}\bar{Q}v_j(y)^{\frac{n+4}{n-4}},\,\,\text{for}\,\,|y|\leq\,\delta M_j^{\frac{2}{n-4}},\\
&\label{maximumpoint}v_j(0)=1,\,\,\nabla_{h_j}v_j(0)=0,\\
&\label{ineqboundoutside}0<v_j(y)\leq \bar{C} |y|^{-\frac{n-4}{2}},\,\,\text{for}\,\,|y|\leq\,\delta M_j^{\frac{2}{n-4}}.
\end{align}
We next show that $v_j$ is uniformly bounded. Since $R_{h_j}>0$ and $R_{u_j^{\frac{4}{n-4}}g}>0$ on $M$, by conformal transformation formula of scalar curvature,
\begin{align}\label{ineqmaxiprinc}
\Delta_{h_j}v_j\leq \frac{(n-4)}{4(n-1)}R_{h_j}v_j,
\end{align}
where $R_{h_j}\to 0$ uniformly in $|y|\leq 2$ as $j\to \infty$. Then the function $\eta_j(y)= (1+|y|^2)^{-1}v_j(y)$ satisfies
\begin{align*}
\Delta_{h_j}\eta_j+ \sum_{k=1}^nb_k(y)\partial_k \eta_j(y)\leq 0,
\end{align*}
in $|y|\leq 2$ with some function $b_k(y)$. By maximum principle,
\begin{align}\label{ineqboundS}
\eta_j(0)\geq \inf_{|y|=r}\eta_j(y)\,\,\,\text{for}\,\,0<r\leq 1.
\end{align}
By the Harnack inequality $(\ref{ineqH})$ in Lemma \ref{lemH},
\begin{align}\label{ineqHv}
\max_{|y|=r}v_j(y) \leq C \min_{|y|=r} v_j(y)\,\,\,\text{for}\,\,0<r\leq 1,
\end{align}
where $C$ is independent of $r$ and $j$. The inequalities $(\ref{ineqboundS})$ and $(\ref{ineqHv})$ immediately derive
\begin{align*}
\max_{|y|=r}v_j(y) \leq C \min_{|y|=r} v_j(y)\leq C v_j(0)=C\,\,\,\text{for}\,\,0<r \leq 1.
\end{align*}
Combining this with $(\ref{ineqboundoutside})$, we have for $|y|\leq\,\delta M_j^{\frac{2}{n-4}}$,
\begin{align*}
v_j(y)\leq C,
\end{align*}
with $C$ independent of $j$, $y$ and $r$.

Standard elliptic estimates of $v_j$ imply that, after possibly passing to a subsequence, $v_j\to v$ in $C^4_{loc}$ in $\mathbb{R}^n$ where by $(\ref{maximumpoint})$ and $(\ref{ineqmaxiprinc})$, $v$ satisfies
\begin{align*}
&\Delta^2v(y)=\frac{n-4}{2}\bar{Q}v^{\frac{n+4}{n-4}},\,\,y\in \mathbb{R}^n,\\
&v(0)=1,\,\,\nabla v(0)=0,\\
&\Delta v(y)\leq 0,\,\,y\in \mathbb{R}^n,\\
&v(y)\geq 0,\,\,y\in \mathbb{R}^n.
\end{align*}
By strong maximum principle, $v(y)>0$ in $\mathbb{R}^n$. Then the classification theorem in \cite{Lin} gives that
\begin{align*}
v(y)=(1+4^{-1}|y|^2)^{-\frac{n-4}{2}}.
\end{align*}
Then the lemma follows.
\end{proof}
\begin{Remark}\label{remark01}
From Lemma \ref{lemlimitmodel}, we can see that the proof of Theorem \ref{thmlowerbound} still works at the isolated blowup point $x_j\to \bar{x}$. Therefore, there exists $C>0$ independent of $j>0$ so that for any isolated blowup point $x_j\to \bar{x}$,
\begin{align*}
u_j(q)\geq C u_j(x_j)^{-1}d_g^{4-n}(q, x_j)
\end{align*}
any $q\in M$ such that $d_g(q, x_j)\geq u_j(x_j)^{-\frac{2}{n-4}}$.
\end{Remark}
We now state the upper bound estimate of $u_j$ near the isolated simple blowup points.
\begin{prop}\label{propupperbound}
Let $(M^n, g)$ be a closed Riemannian manifold of dimension $5\leq n \leq 9$ with $R_g\geq 0$, and also $Q_g\geq 0$ with $Q_g(p_0)>0$ for some point $p_0\in M$. Let $\{u_j\}$ be a sequence of positive solutions to $(\ref{equation1})$ and $x_j\to \bar{x}$ be an isolated simple blowup point. Let $\delta_1$ and $\bar{C}$ be the constants defined in Definition \ref{def02} and $(\ref{boundground})$. Then there exists a constant $C$ depending only on $\delta_1$, $\bar{C}$, $\|R_g\|_{C^1(B_{\delta_1}(\bar{x}))}$ and $\|Q_g\|_{C^1(B_{\delta_1}(\bar{x}))}$ such that
\begin{align}\label{inequpperboundlevel}
u_j(p)\leq C u_j(x_j)^{-1}d_g(p, x_j)^{4-n},\,\,\text{for}\,\,d_g(p, x_j)\leq \frac{\delta_1}{2},
\end{align}
 for $\delta_1>0$ small. Moreover, up to a subsequence,
\begin{align}\label{inequpperboundbehavior}
u_j(x_j)u_j(p)\to a G(\bar{x}, p)+b(p)\,\,\text{in}\,\,C_{loc}^4(B_{\delta_1}(\bar{x})-\{\bar{x}\}),
\end{align}
where $G$ is Green's function of the Paneitz operator $P_g$, $a>0$ is a constant and $b(p)\in C^4(B_{\frac{\delta_1}{2}}(\bar{x}))$ satisfies $P_gb=0$ in $B_{\frac{\delta_1}{2}}(\bar{x})$.
\end{prop}
The proof of the proposition follows after a series of lemmas.

We first give a rough estimate of upper bound of $u_j$ near the isolated simple blowup points.
\begin{lem}\label{lemupperboundestimates1}
Under the condition in Proposition \ref{propupperbound}, assume $R_j\to \infty$ and $0< \epsilon_j<e^{-R_j}$ satisfy $(\ref{ineqmeasurelimit})$ and $(\ref{ineqrulerlimit})$. Denote $M_j=u_j(x_j)$. Then for any small number $0<\sigma<\frac{1}{100}$, there exists $0<\delta_2<\delta_1$ and $C>0$ independent of $j$ such that
\begin{align}
&\label{ineqlembound1}M_j^{\lambda}u_j(p)\leq C d_g(p, x_j)^{4-n+\sigma},\\
&M_j^{\lambda}|\nabla_g^k u_j(p)|\leq C d_g(p, x_j)^{4-n-k+\sigma},
\end{align}
for any $p$ in $R_jM_j^{-\frac{2}{n-4}}\leq d_g(p, x_j)\leq \delta_2$ and $1\leq k\leq 4$, where $\lambda=1-\frac{2}{n-4}\sigma$.
\end{lem}
\begin{proof}
The outline of the proof is from \cite{Li-Zhu}, while the use of our maximum principle here is more subtle. Let $x=(x^1,...,x^n)$ be geodesic normal coordinates centered at $x_j$ for $d_g(p, x_j)\leq \delta$. Let $r=|x|$. For any $0<\delta_2<\delta_1$ to be chosen, let
\begin{align*}
\Omega_j=\{p\in M,\,\,R_jM_j^{-\frac{2}{n-4}}\leq d_g(p, x_j)\leq \delta_2\}.
\end{align*}
We want to use maximum principle to get the upper bound of $u_j$. Before construction of the barrier function on $\Omega_j$, we first go through some properties of $u_j$.

From Lemma \ref{lemlimitmodel}, we know that
\begin{align}\label{ineqoutsidelim}
u_j(p)\leq C R_j^{4-n}M_j,\,\,\text{for}\,\,d_g(p, x_j)=R_jM_j^{-\frac{2}{n-4}},
\end{align}
and there exists a critical point $r_0$ of $\hat{u}_j(r)$ in $0< r < R_jM_j^{-\frac{2}{n-4}}$; moreover, for $r>r_0$, $\hat{u}_j(r)$ is decreasing. By the assumption that $\bar{x}$ is an isolated simple blowup point, $\hat{u}_j$ is strictly decreasing for $R_jM_j^{-\frac{2}{n-4}}< r < \delta_1$. Therefore, combining with the Harnack inequality $(\ref{ineqH})$, for $p \in \Omega_j$ we have
\begin{align*}
d_g(p, x_j)^{\frac{n-4}{2}}u_j(p)&\leq C \bar{u}_j(d_g(p, x_j))\\
&\leq C R_j^{\frac{n-4}{2}}M_j^{-1}\bar{u}_j(R_jM_j^{-\frac{2}{n-4}})\\
&\leq C R_j^{\frac{n-4}{2}}M_j^{-1} R_j^{4-n}M_j\\
&=C R_j^{-\frac{n-4}{2}}.
\end{align*}
This leads to
\begin{align}\label{ineqdecreasing}
u_j(p)^{\frac{8}{n-4}}\leq C R_j^{-4} d_g(p, x_j)^{-4},\,\,\text{for}\,\,R_jM_j^{-\frac{2}{n-4}}< r < \delta_1.
\end{align}
We now define a linear elliptic operator on $\Omega_j$
\begin{align*}
L_j\phi=P_g\phi - \frac{n-4}{2}\bar{Q}u_j^{\frac{8}{n-4}}\phi,\,\,\text{for}\,\,\phi\in C^4(\Omega_j).
\end{align*}
Therefore
\begin{align*}
L_ju_j=0,\,\,\text{in}\,\,\Omega_j.
\end{align*}
Set
\begin{align*}
\varphi(p)=B\bar{M}_j\delta_2^{\sigma}d_g(p, x_j)^{-\sigma}+A M_j^{-1+\frac{2}{n-4}\sigma}d_g(p, x_j)^{-n+4+\sigma},\,\,p\in \Omega_j,
\end{align*}
where $A, B>0$ are constant to be determined, $0< \sigma <\frac{1}{100}$ and
\begin{align*}
\bar{M}_j=\sup_{d_g(p, x_j)=\delta_2}u_j\,\leq \bar{C} \delta_2^{-\frac{n-4}{2}}.
\end{align*}
There exists $C>0$, for $m>0$, $1\leq k\leq 4$ and any $p\in M$ fixed and $q \in M$ so that $d_g(p, q)<\delta_2$ with $\delta_2$ less than the injectivity radius,
\begin{align}\label{ineqdistance}
|D_g^k d_g(p, q)^{-m}|\leq C m^k d_g(p, q)^{-m-k}.
\end{align}
It is easy to check that there exists $\delta_2>0$ independent of $j$ so that in $\Omega_j$
\begin{align*}
&|(P_g-\Delta_0^2)|x|^{-\sigma}|\leq 100^{-1}|P_g (|x|^{-\sigma})|,\\
&|(P_g-\Delta_0^2)|x|^{-n+4+\sigma}|\leq 100^{-1}|P_g(|x|^{-n+4+\sigma})|,
\end{align*}
where $|x|=d_g(p, x_j)$ and $\Delta_0$ is the Euclidean Laplacian in the normal coordinates.
It is easy to check that for $0<m< n-4$ and $0< r < \delta_2$,
\begin{align}\label{ineqordermainterm}
&-\Delta_0 r^{-m}=-m(m+2-n)r^{-m-2}>0\\
&\Delta_0^2r^{-m}=m(m+2-n)(m+2)(m+4-n)r^{-m-4}>0.
\end{align}
But for $p\in \Omega_j$, by $(\ref{ineqdecreasing})$
\begin{align*}
\frac{n-4}{2}\bar{Q}u_j(p)^{\frac{8}{n-4}}r^{-m}\leq \frac{n-4}{2}\bar{Q}C R_j^{-4}r^{-m-4}
\end{align*}
Therefore,
\begin{align*}
L_j\varphi_j\geq 0\,\,\text{in}\,\,\Omega_j
\end{align*}
for $j$ large. By $(\ref{ineqoutsidelim})$, for $A>1$,
\begin{align}\label{ineqinnerboundary}
u_j(p)< \varphi_j(p),\,\,\text{for}\,\,d_g(p, x_j)=R_jM_j^{-\frac{2}{n-4}}.
\end{align}
Also, for $B>1$,
\begin{align}\label{ineqouterboundary}
u_j(p)< \varphi_j(p),\,\,\text{for}\,\,d_g(p, x_j)=\delta_2.
\end{align}
We now want to check sign of the scalar curvature $R_{(\varphi_j - u_j)^{\frac{4}{n-4}}g}$ near $\partial \Omega_j$. By conformal transformation formula, it has the same sign as
\begin{align*}
-\frac{4(n-1)}{(n-4)}\Delta_g (\varphi_j - u_j)- \frac{8(n-1)}{(n-4)^2}\frac{|\nabla_g(\varphi_j - u_j)|^2}{(\varphi_j - u_j)}+R_g(\varphi_j - u_j)
\end{align*}
Combining $(\ref{boundground})$ and standard interior estimate of $(\ref{equation1})$, we have for $k=1,\,2$,
\begin{align}\label{ineqboundary1}
|D_g^ku_j(p)|\leq C d_g(p, x_j)^{-\frac{n-4}{2}-k}
\end{align}
for some constant independent of $j$, where $p\in \Omega_j$.
It is easy to check that for $0<m<n-4$,
\begin{align}\label{ineqboundary2}
\Delta_0|x|^{-m}+\frac{2}{n-4} \frac{|\nabla_0 |x|^{-m}|^2}{|x|^{-m}}&=\big(m(m+2-n)\,+\,\frac{2m^2}{n-4})|x|^{-m-2}\\
&=\frac{m(n-2)(m-(n-4))}{n-4}|x|^{-m-2}<0.
\end{align}
Also, note that for any positive functions $\phi_1,\,\phi_2\in C^2$, it holds that
\begin{align}\label{ineqboundary3}
\Delta_0(\phi_1+\phi_2)+\frac{2}{n-4} \frac{|\nabla_0 (\phi_1+\phi_2)|^2}{\phi_1+\phi_2}\leq (\Delta_0 \phi_1 +\frac{2}{n-4} \frac{|\nabla_0 (\phi_1)|^2}{\phi_1})+(\Delta_0 \phi_2 +\frac{2}{n-4} \frac{|\nabla_0 (\phi_2)|^2}{\phi_2}).
\end{align}
Here we have used the fact that for $a,b,c,d>0$
\begin{align*}
&\frac{2c\,d}{a+b}\leq \frac{b\,c^2}{a(a+b)}+\frac{a\,d^2}{b(a+b)},\,\,\text{so that}\\
&\frac{(c+d)^2}{a+b}=\frac{c^2+2c\,d+d^2}{a+b}\leq \frac{c^2}{a}+\frac{d^2}{b}.
\end{align*}
Using $(\ref{ineqinnerboundary})$-$(\ref{ineqboundary2})$ and $(\ref{ineqboundary3})$, we can choose $A, B>100^n(1+ C)$ independent of $j$ and $t$ with $C>0$ in $(\ref{ineqboundary1})$ so that
\begin{align}\label{ineqcurvatureboundary}
-\frac{4(n-1)}{(n-4)}\Delta_g (t\varphi_j - u_j)- \frac{8(n-1)}{(n-4)^2}\frac{|\nabla_g(t\varphi_j - u_j)|^2}{(t\varphi_j - u_j)}+R_g(t\varphi_j - u_j)>0\,\,\,\text{on}\,\,\partial \Omega_j,
\end{align}
for all $t\geq 1$. For $t\geq 1$, we define
\begin{align*}
\phi_j^t(p)=t\varphi_j(p)-u_j(p),\,\,p\in \Omega_j.
\end{align*}
Then
\begin{align}\label{ineqlinearineqn}
0\leq L_j \phi_j^t=P_g\phi_j^t - \frac{n-4}{2}\bar{Q}\phi_j^t\,\,\,\,\text{in} \,\,\Omega_j.
\end{align}
If
\begin{align}\label{inequpperbound1}
\phi_j^1=\varphi_j-u_j\geq 0\,\,\,\,\text{in}\,\,\Omega_j,
\end{align}
 then we are done. Else, since $\Omega_j$ is compact, we pick up the smallest number $t_j>1$ so that $\phi_j^{t_j}\geq 0$. Therefore, by $(\ref{ineqlinearineqn})$
 \begin{align}\label{ineqlinearineqn1}
 P_g\phi_j^{t_j}\geq \frac{n-4}{2}\bar{Q}\phi_j^{t_j}\geq 0.
 \end{align}
 Combining with $(\ref{ineqinnerboundary})$, $(\ref{ineqouterboundary})$, $(\ref{ineqcurvatureboundary})$ and $(\ref{ineqlinearineqn1})$, the maximum principle in Lemma \ref{lemmaximumprinciple} implies
\begin{align*}
\phi_j^{t_j}>0\,\,\text{in}\,\,\Omega_j,
\end{align*}
contradicting with the choice of $t_j$. Therefore, $(\ref{inequpperbound1})$ holds. Now for $p \in \Omega_j$, we use Lemma \ref{lemH}, monotonicity of $\hat{u}_j$, and apply $(\ref{inequpperbound1})$ at $p$ to obtain
\begin{align*}
\delta_2^{\frac{n-4}{2}}\bar{M}_j&\leq C \hat{u}_j(\delta_2)\leq C \hat{u}_j(d_g(p, x_j))\\
&\leq Cd_g(p, x_j)^{\frac{n-4}{2}}(B\bar{M}_j\delta_2^{\sigma}d_g(p, x_j)^{-\sigma}+A M_j^{-\lambda}d_g(p, x_j)^{4-n+\delta}).
\end{align*}
Here $\frac{n-4}{2}> \sigma$. We choose $p$ with $d_g(p, x_j)$ to be a small fixed number depending on $n, \sigma, \delta_2$ to obtain
\begin{align*}
\bar{M}_j\leq C(n, \sigma, \delta_2)M_j^{-\lambda}.
\end{align*}
 Therefore, the inequality $(\ref{ineqlembound1})$ is then established from $(\ref{inequpperbound1})$, and based on standard interior estimates for derivatives of $u_j$, the lemma is proved.
\end{proof}

\begin{lem}\label{lemupperboundsphere}
Under the assumption in Proposition \ref{propupperbound}, for any $0< \rho\leq \frac{\delta_2}{2}$ there exists a constant $C(\rho)>0$ such that
\begin{align*}
\limsup_{j\to \infty}\max_{p\in \partial B_{\rho}(x_j)} u_j(p) M_j \leq C(\rho).
\end{align*}
where $M_j=u_j(x_j)$.
\end{lem}
\begin{proof}
By Lemma \ref{lemH}, it suffices to show the inequality for some fixed small constant $\rho>0$.

For any $p_{\rho}\in \partial B_{\rho}(x_j)$, we denote $\xi_j(p)=u_j(p_{\rho})^{-1}u_j(p)$. Then $\xi_j$ satisfies
\begin{align*}
P_g\xi_j(p)=\frac{n-4}{2}\bar{Q}u_j(p_{\rho})^{\frac{8}{n-4}}\xi_j(p)^{\frac{n+4}{n-4}}.
\end{align*}
For any compact subset $K \subseteq B_{\frac{\delta_2}{2}}(\bar{x})- \{\bar{x}\}$, there exists $C(K)>0$ such that for $j$ large
\begin{align*}
C(K)^{-1}\leq \xi_j\leq C(K)\,\,\text{in}\,\,K.
\end{align*}
Moreover, by Lemma \ref{lemH}, there exists $C>0$ independent of $0< r< \delta_2$ and $j$ so that
\begin{align}\label{ineqHsphere1}
\max_{B_r(x_j)- B_{\frac{r}{2}}(x_j)}u_j\leq C\inf_{B_r(x_j)- B_{\frac{r}{2}}(x_j)}u_j.
\end{align}
By the estimates $(\ref{ineqlembound1})$, $u_j(p_{\rho})\to 0$ as $j\to \infty$. Therefore, by interior estimates of $\xi_j$ , up to a subsequence,
\begin{align*}
\xi_j\to \xi\,\,\text{in}\,\,C_{loc}^4(B_{\frac{\delta_2}{2}}(\bar{x})- \{\bar{x}\}),
\end{align*}
with $\xi>0$ such that
\begin{align*}
P_g\xi =0\,\,\text{in}\,\,B_{\frac{\delta_2}{2}}(\bar{x})- \{\bar{x}\},
\end{align*}
and $\xi$ satisfies $(\ref{ineqHsphere1})$ for $0<r<\frac{\delta_2}{2}$. Moreover, for $0<r<\rho$ and $\bar{\xi}(r)=|\partial B_r|^{-1}\int_{\partial B_r(\bar{x})}\xi ds_g$,
\begin{align*}
\lim_{j\to\infty}u_j(p_{\rho})^{-1}r^{\frac{n-4}{2}}\bar{u}_j(r)=r^{\frac{n-4}{2}}\bar{\xi}(r).
\end{align*}
Since $x_j\to \bar{x}$ is an isolated simple blowup point, $r^{\frac{n-4}{2}}\bar{\xi}(r)$ is non-increasing in $0<r<\rho$. Therefore, $\bar{x}$ is not a regular point of $\xi$.

Recall that
\begin{align*}
-\frac{4(n-1)}{n-2}\Delta_gu_j^{\frac{n-2}{n-4}}+R_gu_j^{\frac{n-2}{n-4}}=R_{u_j^{\frac{4}{n-4}}g}u_j^{\frac{n+2}{n-4}}\geq 0.
\end{align*}
Passing to the limit, we have
\begin{align}
-\frac{4(n-1)}{n-2}\Delta_g\xi^{\frac{n-2}{n-4}}+R_g\xi^{\frac{n-2}{n-4}}\geq 0,
\end{align}
in $B_{\frac{\delta_2}{2}}(\bar{x})- \{\bar{x}\}$.

 For later use, if $Q_g$ is not pointwisely non-negative in $B_{\rho}(\bar{x})$, by Theorem \ref{thm1} we choose $\tilde{g}=\phi^{-\frac{4}{n-4}}g$ so that $R_{\tilde{g}}>0$ and $Q_{\tilde{g}}\geq 0$ in $M$. Then $\tilde{\xi}=\phi \xi$ is still singular at $\bar{x}$ and all above information for the limit holds with $\xi$ and $g$ replaced by $\tilde{\xi}$ and $\tilde{g}$. So from now on we assume that $R_{g}>0$ and $Q_g\geq 0$.

From Corollary \ref{cor9.1}, for $\rho>0$ small, there exists $m>0$ independent of $j$ such that for $j$ large
\begin{align}\label{ineqint1}
\int_{B_{\rho}(x_j)}(P_g\xi_j-\frac{n-4}{2}Q_g\xi_j) d V_g&=\int_{\partial B_{\rho}(x_j)}\big(\frac{\partial}{\partial \nu}\Delta_g \xi_j-(a_nR_g\frac{\partial}{\partial \nu} \xi_j-b_nRic_g(\nabla_g \xi_j, \nu))\big) d s_g\\
&=\int_{\partial B_{\rho}(x_j)}\big(\frac{\partial}{\partial \nu}\Delta_g \xi-(a_nR_g\frac{\partial}{\partial \nu} \xi-b_nRic_g(\nabla_g \xi, \nu))\big) d s_g+o(1)>m.
\end{align}
On the other hand, nonnegativity of $Q_g$ implies
\begin{align}\label{ineqint2}
\int_{B_{\rho}(x_j)}(P_g\xi_j-\frac{n-4}{2}Q_g\xi_j) d V_g&=\int_{B_{\rho}(x_j)}(\frac{n-4}{2}\bar{Q}u_j(p_{\rho})^{-1}u_j(p)^{\frac{n+4}{n-4}} - \frac{n-4}{2}Q_g\xi_j)\,d V_g\\
&\leq \frac{n-4}{2}\bar{Q}\int_{B_{\rho}(x_j)}u_j(p_{\rho})^{-1}u_j(p)^{\frac{n+4}{n-4}} \,d V_g.
\end{align}
Using $(\ref{ineqmeasurelimit})$ and $\epsilon_j\leq e^{-R_j}$, we have
\begin{align*}
\int_{B_{R_jM_j^{-\frac{2}{n-4}}}(x_j)}u_j^{\frac{n+4}{n-4}}dV_g\leq C M_j^{-1},
\end{align*}
while by $(\ref{ineqlembound1})$ we have
\begin{align*}
\int_{B_{\rho}(x_j)-B_{R_jM_j^{-\frac{2}{n-4}}}(x_j)}u_j^{\frac{n+4}{n-4}} d V_g &\leq C \int_{B_{\rho}(x_j)-B_{R_jM_j^{-\frac{2}{n-4}}}(x_j)}(M_j^{-\lambda} d_g(p, x_j)^{4-n+\sigma})^{\frac{n+4}{n-4}}\\
&\leq C (R_jM_j^{-\frac{2}{n-4}})^{-4+\frac{n+4}{n-4}\sigma}M_j^{-\lambda\frac{n+4}{n-4}}\\
&=R_j^{-4+\frac{n+4}{n-4}\sigma}M_j^{-1}=o(1)M_j^{-1}.
\end{align*}
Therefore,
\begin{align}\label{ineqint3}
\int_{B_{\rho}(x_j)}u_j^{\frac{n+4}{n-4}} d V_g\leq C M_j^{-1}.
\end{align}
Lemma \ref{lemupperboundsphere} follows from the inequalities $(\ref{ineqint1})$-$(\ref{ineqint3})$.
\end{proof}

\noindent{\it Proof of Proposition \ref{propupperbound}.}\quad Suppose $(\ref{inequpperboundlevel})$ fails. Let $M_j=u_j(x_j)$. Then there exists a subsequence $u_j$ and $\{p_j\}$ with $d_g(p_j, x_j)\leq \frac{\delta_2}{2}$ with $\delta_2$ in Lemma \ref{lemupperboundestimates1} such that
\begin{align}\label{ineqdivergence}
u_j(p_j)M_jd_g(p_j, x_j)^{n-4}\to \infty.
\end{align}
 If $Q_g\leq 0$ does not hold in $B_{\delta_2}(\bar{x})$, let $g_0=\phi^{-\frac{4}{n-4}}g$ be the metric so that $Q_{g_0}\geq 0$ and $R_{g_0}>0$ on $M$. Then $(\ref{ineqdivergence})$ holds for $g$ and $u_j$ replaced by $g_0$ and $\tilde{u}_j=\phi u_j$. Therefore, from now on we assume that $Q_g\geq 0$ and $R_g>0$ on $M$. By Lemma \ref{lemlimitmodel} and $0<\epsilon_j\leq e^{-R_j}$,
 \begin{align*}
 R_jM_j^{-\frac{2}{n-4}}\leq d_g(p_j, x_j) \leq \frac{\delta_2}{2}.
 \end{align*}
 Let $x=(x^1,...,x^n)$ be the geodesic normal coordinates centered at $x_j$. Denote $y=d_j^{-1}x$ where $d_j=d_g(p_j, x_j)$.  We do rescaling
 \begin{align*}
 v_j(y)=d_j^{\frac{n-4}{2}}u_j(\exp_{x_j}(d_jy)),\,\,|y|\leq 2.
 \end{align*}
 Then $v_j$ satisfies
 \begin{align*}
 P_{h_j}v_j(y)=\frac{n-4}{2}\bar{Q}v_j(y)^{\frac{n+4}{n-4}},\,\,|y|\leq 2,
 \end{align*}
 where $h_j=d_j^{-2}g_j$ so that $(h_j)_{pq}(y)=(g)_{pq}(d_jy)$. The metrics $h_j$ depend on $j$. But since $d_j$ has uniform upper bound, the sequence of metrics stays in compact sets with strong norms and all the results in Lemma \ref{lemupperboundsphere} hold uniformly for $j$. Also, the conclusion of Lemma \ref{lemupperboundestimates1} is scaling invariant. Note that as the metrics $h_j$ converge to $h$, Green's functions of Paneitz operators $P_{h_j}$ converge to Green's functions of Paneitz operators $P_{h}$ uniformly away from the singularity. In particular, if $d_j\to 0$ then $h_j$ converges to a flat metric on $B_2(0)$ so that in proof of Proposition \ref{propsingularity}, $G(p, \bar{x})$ will be replaced by $c_n|y|^{4-n}$ in Euclidean balls with $c_n$ in $(\ref{expansion1})$. Therefore, Lemma \ref{lemupperboundsphere} holds for $v_j$ so that
 \begin{align*}
 \max_{|x|=1}v_j(0)v_j(x)\leq C,
 \end{align*}
which shows that
\begin{align*}
M_ju_j(p_j)d_g(p_j, x_j)^{4-n}\leq C,
\end{align*}
contradicting with $(\ref{ineqdivergence})$. We have proved $(\ref{inequpperboundlevel})$ in $B_{\frac{\delta_2}{2}}(\bar{x})$. By Lemma \ref{lemH} the inequality $(\ref{inequpperboundlevel})$ holds in $B_{\delta_1}(\bar{x})$.\vskip0.2cm

The same properties for $\xi_j$ in Lemma \ref{lemupperboundsphere} now hold for $M_j u_j$ in $B_{\frac{\delta_2}{2}}(\bar{x})$. Up to a subsequence
\begin{align*}
M_ju_j\to v\,\,\text{in}\,\,C_{loc}^4(B_{\frac{\delta_2}{2}}(\bar{x}))
\end{align*}
where
\begin{align*}
P_gv=0\,\,\text{in}\,\,B_{\frac{\delta_2}{2}}(\bar{x}).
\end{align*}
By Remark \ref{remark01}, $v>0$ in $B_{\frac{\delta_2}{2}}(\bar{x})$. Since $\bar{x}$ is an isolated simple blowup point, the same argument in Lemma \ref{lemupperboundsphere} shows that $r^{\frac{n-4}{2}}\bar{v}(r)$ is non-increasing for $0<r<\frac{\delta_2}{2}$, where $\bar{v}(r)=|\partial B_r(\bar{x})|^{-1}\int_{\partial B_r(\bar{x})}v ds_g$. Combining with the Harnack inequality, it implies that $v$ is not regular at $\bar{x}$. Also, $v$ satisfies the condition in Proposition \ref{propsingularity}. By Proposition \ref{propsingularity}, we obtain $(\ref{inequpperboundbehavior})$. This completes the proof of Proposition \ref{propupperbound}.

\qed
\vskip0.2cm
As an easy consequence of Proposition \ref{propupperbound} and by standard interior estimates of the elliptic equation $(\ref{equation1})$, we have
\begin{cor}\label{corerrorterms}
Under the condition in Lemma \ref{lemupperboundestimates1}, there exists $\delta_2>0$ independent of $j$ such that for $R_jM_j^{-\frac{2}{n-4}}\leq d_g(p, x_j)\leq \delta_2$
\begin{align}\label{ineqboundout1}
|\nabla_g^k u_j(p)|\leq\,\,C\,M_j^{-1}d_g(p, x_j)^{4-n-k}\,\,\,\text{for}\,\,0\leq k\leq 4,
\end{align}
where $M_j=u_j(x_j)$, and $C$ is a constant independent of $j$. Let $x$ be geodesic normal coordinates of $(\Omega ,g)$ centered at $x_j$. Then for any fixed $r\leq \delta_2$, there exists $C>0$ depending on $|g|_{C^3(\Omega)}$ such that

\begin{align}\label{ineqbounderrorterm2}
|\int_{d_g(p, x_j)\leq r}(x\cdot \nabla u +\frac{n-4}{2}u)(\Delta^2 - P_g)u dx|\leq CM_j^{-\frac{4}{n-4}+o(1)}
\end{align}
with the term $o(1)\to 0$ as $j\to \infty$.

\end{cor}

\begin{proof}
Inequality $(\ref{ineqboundout1})$ is a direct consequence of Proposition  \ref{propupperbound} and standard interior estimates of the elliptic equation $(\ref{equation1})$. We will next establish $(\ref{ineqbounderrorterm2})$. Note that $0<\epsilon_j \leq e^{-R_j}$. Using the estimates $(\ref{ineqboundout1})$, $(\ref{ineqmeasurelimit})$ and $(\ref{ineqrulerlimit})$, and recall the error bound $(\ref{ineqbounderrorterms})$, we have
\begin{align*}
&\int_{|x|\leq R_jM_j^{-\frac{2}{n-4}}}\,|(x\cdot \nabla u +\frac{n-4}{2}u)(\Delta^2 - P_g)u|\,dx\\
&\leq \int_{|x|\leq R_jM_j^{-\frac{2}{n-4}}}\,C(|x||Du(x)|+ u(x))(|x|^2|D^4u(x)|+|x| \,|D^3 u(x)|+|D^2u(x)|+|D u(x)|+u(x)) dx\\
&\leq C\int_{|y|\leq R_j} M_j(1+4^{-1}|y|^2)^{-\frac{n-4}{2}}M_j(1+4^{-1}|y|^2)^{-\frac{n-4}{2}-1}M_j^{\frac{4}{n-4}} M_j^{-\frac{2n}{n-4}} dy\\
&=CM_j^{-\frac{4}{n-4}}\int_{|y|\leq R_j}(1+4^{-1}|y|^2)^{3-n}dy=CM_j^{-\frac{4}{n-4}+o(1)},\,\,\,\,\text{and}\\
&\int_{R_jM_j^{-\frac{2}{n-4}} \leq |x|\leq r}\,|(x\cdot \nabla u +\frac{n-4}{2}u)(\Delta^2 - P_g)u|\,dx\\
&\leq \int_{R_jM_j^{-\frac{2}{n-4}} \leq |x|\leq r}C(|x||Du(x)|+ u(x))(|x|^2|D^4u(x)|+|x| \,|D^3 u(x)|+|D^2u(x)|+|D u(x)|+u(x)) dx\\
&\leq C\, \int_{R_jM_j^{-\frac{2}{n-4}} \leq |x|\leq r}\,M_j^{-2}|x|^{6-2n}\,dx\\
&\leq CM_j^{-\frac{4}{n-4}+o(1)},
\end{align*}
where the term $o(1)\to 0$ as $j\to \infty$ and $C>0$ is a constant depending on $|g|_{C^3(\Omega)}$. Therefore,
\begin{align*}
\int_{|x|\leq r}\,|(x\cdot \nabla u +\frac{n-4}{2}u)(\Delta^2 - P_g)u|\,dx\leq CM_j^{-\frac{4}{n-4}+o(1)}\,\,\text{for}\,\,R_jM_j^{-\frac{2}{n-4}}\leq r,
\end{align*}
where $C>0$ is a constant independent of $j$ and the term $o(1)\to 0$ as $j\to \infty$.

For $n=5$, it is good for the discussion to come. For $n\geq 6$, better estimate is needed in order to cancel the error terms in the Pohozaev identity.
By $(\ref{ineqmeasurelimit})$,
\begin{align*}
u_j(\exp_{x_j}(x))\leq 2M_j(1+4^{-1}M_j^{\frac{4}{n-4}}|x|^2)^{-\frac{n-4}{2}},\,\,\text{for}\,\,|x|\,\leq\,R_jM_j^{-\frac{2}{n-4}}.
\end{align*}
Combining with Proposition \ref{propupperbound}, we have
\begin{align*}
u_j(\exp_{x_j}(x))&\leq C \min\{ M_j(1+4^{-1}M_j^{\frac{4}{n-4}}|x|^2)^{-\frac{n-4}{2}},\,CM_j^{-1}|x|^{4-n}\}\\
&\leq C\,M_j(1+4^{-1}M_j^{\frac{4}{n-4}}|x|^2)^{-\frac{n-4}{2}},\,\,\,\text{for}\,\,|x|\,\leq\,\delta_2.
\end{align*}
For $n=6$,
\begin{align*}
&\int_{|x|\leq r}\,|(x\cdot \nabla u +\frac{n-4}{2}u)(\Delta^2 - P_g)u|\,dx\\
&\leq C\int_1^{M_j^{\frac{2}{n-4}}r}M_j^{-2}M_j^{\frac{2(n-6)}{n-4}}|y|^{5-n}d |y|\\
&\leq CM_j^{-\frac{4}{n-4}}\ln(M_j^{\frac{2}{n-4}}r),\,\,\text{for}\,\,R_jM_j^{-\frac{2}{n-4}}\leq r,
\end{align*}
For $n\geq 7$,
\begin{align*}
&\int_{|x|\leq r}\,|(x\cdot \nabla u +\frac{n-4}{2}u)(\Delta^2 - P_g)u|\,dx\\
&\leq C\int_1^{M_j^{\frac{2}{n-4}}r}M_j^{-2}M_j^{\frac{2(n-6)}{n-4}}|y|^{5-n}d |y|\\
&\leq CM_j^{-\frac{4}{n-4}},\,\,\text{for}\,\,R_jM_j^{-\frac{2}{n-4}}\leq r,
\end{align*}

For the term $M_j^2\int_{|x|\leq r}|Q_g|\,(u_j^2+|x|\,|Du_j|\,u_j) d x$ with $r>0$ fixed,
\begin{align*}
M_j^2\int_{|x|\leq r}|Q_g|\,(u_j^2+|x|\,|Du_j|\,u_j)\, d x&\leq C\,M_j^2\int_0^{rM_j^{\frac{2}{n-4}}}M_j^2(1+|y|)^{8-2n}M_j^{-\frac{2n}{n-4}}|y|^{n-1}d|y|\\
&\leq C\,M_j^{2-\frac{8}{n-4}}\int_0^{rM_j^{\frac{2}{n-4}}}\,(1+|y|)^{7-n}d|y|
\end{align*}
For $n=6$,
\begin{align*}
M_j^2\int_{|x|\leq r}|Q_g|\,(u_j^2+|x|\,|Du_j|\,u_j)\, d x\leq \,C\,r^2.
\end{align*}
For $n=7$,
\begin{align*}
M_j^2\int_{|x|\leq r}|Q_g|\,(u_j^2+|x|\,|Du_j|\,u_j)\, d x\leq \,C\,r.
\end{align*}
These are good terms. For later use, estimates on the term $M_j^2\int_{|x|\leq r}\,(u_j+|x|\,|Du_j|)\,|D^2u_j|\, d x$ will be needed for  $n=6$; while for $n=7$, estimates on the term $M_j^2\int_{|x|\leq r}\,(u_j+|x|\,|Du_j|)\,(|Du_j|+|D^2u|)\, d x$ is needed.

\end{proof}

\begin{prop}\label{propinequality}
Let $(M^n, g)$ be a closed Riemannian manifold of dimension $n=5$ with $R_g\geq 0$, and also $Q_g\geq 0$ with $Q_g(p_0)>0$ for some point $p_0\in M$. Let $\{u_j\}$ be a sequence of positive solutions to $(\ref{equation1})$ and $x_j\to \bar{x}$ be an isolated simple blowup point so that
\begin{align*}
u_j(x_j)u_j(p)\to h(p)\,\,\,\,\text{in}\,\,C_{loc}^{4,\alpha}(B_{\delta_2}(\bar{x})-\{\bar{x}\}),
\end{align*}
for some $0< \alpha <1$. Assume that for some constants $a>0$ and $A$,
\begin{align*}
h(p)=\frac{a}{d_g(p, \bar{x})^{n-4}}+A+o(1)\,\,\text{as}\,\,d_g(p, \bar{x})\to 0.
\end{align*}
Then $A= 0$.
\end{prop}
\begin{proof}
Let $x=(x^1,...,x^n)$ be geodesic normal coordinates at $x_j$. Denote $\Omega_{\gamma, j}=B_{\gamma}(x_j)$ for $\gamma< \frac{\delta_2}{2}$. Then $\Omega_{\gamma, j}\to \Omega_{\gamma}=B_{\gamma}(\bar{x})$.  By the Pohozaev identity,
\begin{align*}
&\int_{\partial \Omega_{\gamma, j}} \,\frac{n-4}{2}(u_j\frac{\partial}{\partial \nu}(\Delta u_j)-\Delta u_j  \frac{\partial}{\partial \nu}u_j)+ ((x\cdot\nabla u_j)\frac{\partial}{\partial \nu}(\Delta u_j)- \frac{\partial}{\partial \nu} (x\cdot\nabla u_j)\Delta u_j + \frac{1}{2} (\Delta u_j)^2x\cdot \nu) ds\\
&=\int_{\Omega_{\gamma, j}}(x\cdot \nabla u_j+ \frac{n-4}{2}u_j)(\Delta^2-P_g)u_j\,dx\,+\,\frac{(n-4)^2}{4n}\bar{Q}\int_{\partial \Omega_{\gamma, j}}(x\cdot \nu)u_j^{\frac{2n}{n-4}}\, dx.
\end{align*}
Multiplying $M_j^2=u_j(x_j)^2$ on both sides of the identity and taking limit $\lim_{\gamma\to 0^+}\,\limsup_{j\to \infty}$ on both sides, we have that by Corollary \ref{corerrorterms},
\begin{align*}
\lim_{\gamma\to 0}\limsup_{j\to \infty} M_j^2 \int_{\Omega_{\gamma, j}}(x\cdot \nabla u_j+ \frac{n-4}{2}u_j)(\Delta^2-P_g)u_j\,dx=0,
\end{align*}
and
\begin{align*}
&\lim_{\gamma\to 0}[\int_{\partial \Omega_{\gamma}} \,\frac{n-4}{2}(h\frac{\partial}{\partial \nu}(\Delta h)-\Delta h \frac{\partial}{\partial \nu}h)+ ((x\cdot\nabla h)\frac{\partial}{\partial \nu}(\Delta h)- \frac{\partial}{\partial \nu} (x\cdot\nabla h)\Delta h + \frac{1}{2} (\Delta h)^2x\cdot \nu) ds]\\
&=\lim_{\gamma\to 0}\limsup_{j\to \infty}\,\,M_j^2\int_{\partial \Omega_{\gamma, j}} \,[\frac{n-4}{2}(u_j\frac{\partial}{\partial \nu}(\Delta u_j)-\Delta u_j  \frac{\partial}{\partial \nu}u_j)+ ((x\cdot\nabla u_j)\frac{\partial}{\partial \nu}(\Delta u_j)\\
&- \frac{\partial}{\partial \nu} (x\cdot\nabla u_j)\Delta u_j + \frac{1}{2} (\Delta u_j)^2x\cdot \nu)] ds\\
&=\lim_{\gamma\to 0}\limsup_{j\to \infty}\frac{(n-4)^2}{4n}\bar{Q} M_j^{-\frac{8}{n-4}} \int_{\partial \Omega_{\gamma, j}}(x\cdot \nu)(M_j\,u_j)^{\frac{2n}{n-4}}\, dx=\,0.
\end{align*}
By assumption,
\begin{align*}
&\lim_{\gamma\to 0}[\int_{\partial \Omega_{\gamma}} \,\frac{n-4}{2}(h\frac{\partial}{\partial \nu}(\Delta h)-\Delta h \frac{\partial}{\partial \nu}h)+ ((x\cdot\nabla h)\frac{\partial}{\partial \nu}(\Delta h)- \frac{\partial}{\partial \nu} (x\cdot\nabla h)\Delta h + \frac{1}{2} (\Delta h)^2x\cdot \nu) ds]\\
&=\lim_{\gamma\to 0}\,\int_{\partial \Omega_{\gamma}}(n-4)^2(n-2)a A |x|^{1-n} d s\\
&=(n-4)^2(n-2)a A |\mathbb{S}^{n-1}|,
\end{align*}
where $|\mathbb{S}^{n-1}|$ is area of $(n-1)-$dimensional round sphere.
Therefore,
\begin{align*}
A= 0.
\end{align*}
\end{proof}
\begin{Remark}\label{remark6}
Corollary \ref{corerrorterms} is not enough to prove Proposition \ref{propinequality} for manifolds of dimension $n=6$ and $n=7$. Let $U_0(r)=(1+4^{-1}r^2)^{-\frac{n-4}{2}}$ be a solution to
\begin{align}\label{equation123}
\Delta^2U_0=\frac{n-4}{2}\bar{Q}U_0^{\frac{n+4}{n-4}}
\end{align}
on $\mathbb{R}^n$ with dimension $n=6$ or $7$. The linearized equation of $(\ref{equation123})$ is
\begin{align}\label{equation1231}
\Delta^2\phi(y)=\frac{n+4}{2}\bar{Q}U_0(y)^{\frac{8}{n-4}}\phi(y)
\end{align}
for $y\in \mathbb{R}^n$.
As in \cite{Marques}, if we can show that for any solution $\phi$ to $(\ref{equation1231})$ with $\phi(y)\to 0$ as $y\to \infty$ it holds that
\begin{align*}
\phi(z)=c_0(z\cdot \nabla U_0(z) + \frac{n-4}{2}U_0(z))+ \sum_{j=1}^nc_j\partial_{z_j}U_0(z),\,\,z\in \mathbb{R}^n,
\end{align*}
with $c_0,\,...,\, c_n$ some constant, then we can prove that $|v_j(y)- U_0(y)|\leq C M_j^{-2}$ for $|y|\leq rM_j^{\frac{2}{n-4}}$ where $v_j(y)=M_j^{-1}u_j(M_j^{-\frac{2}{n-4}}y)$ and $C>0$ is a constant independent of $j$. This combining with Green's representation leads to the better estimate
\begin{align}
|\int_{d_g(p, x_j)\leq r}(x\cdot \nabla u +\frac{n-4}{2}u)(\Delta^2 - P_g)u dx|= o(1)M_j^{-2},
\end{align}
in conformal normal coordinates and the corresponding conformal metric $g$. Then Proposition \ref{propinequality} still holds for $n=6$ and $n=7$ in conformal normal coordinates and the corresponding conformal metric $g$.
\end{Remark}

\section{From isolated blowup points to isolated simple blowup points}
In this section we show that an isolated blowup point is an isolated simple blowup point.
\begin{prop}\label{propisolatedsingularpoints}
Let $(M^n, g)$ be a closed Riemannian manifold of dimension $n=5$ 
 with $R_g\geq 0$, and also $Q_g\geq 0$ with $Q_g(p_0)>0$ for some point $p_0\in M$. Let $\{u_j\}$ be a sequence of positive solutions to $(\ref{equation1})$ and $x_j\to \bar{x}$ be an isolated blowup point. Let $M_j=u_j(x_j)$. Then $\bar{x}$ is an isolated simple blow up point.
\end{prop}
\begin{proof}
We prove the Proposition by contradiction argument. Assume that $\bar{x}$ is not an isolated simple blow up point. Then there exists two critical point of $r^{\frac{n-4}{2}}\bar{u}_j(r)$ in $(0, \mu_j)$ with $\mu_j\to 0$ up to as subsequence as $j\to \infty$. By Lemma \ref{lemlimitmodel} and let $0<\epsilon_j<e^{-R_j}$, $r^{\frac{n-4}{2}}\bar{u}_j(r)$ has precisely one critical point in $(0, R_jM_j^{-\frac{2}{n-4}})$. We choose $\mu_j$ to be the second critical point of $r^{\frac{n-4}{2}}\bar{u}_j(r)$ so that $\mu_j\geq R_jM_j^{-\frac{2}{n-4}}$ and by assumption $\mu_j\to 0$.

Let $x=(x^1,...,x^n)$ be the geodesic normal coordinates centered at $x_j$, and let $y=\mu_j^{-1}x$. For simple notations, we assume $\delta_2=1$. We define the scaled metric $h_j=\mu_j^{-2}g$ so that $(h_j)_{pq}(\mu_j^{-1} x)dx^pdx^q=g_{pq}(x)dx^pdx^q$, and
\begin{align*}
\xi_j(y)=\mu_j^{\frac{n-4}{2}}u_j(\exp_{x_j}(\mu_j y)),\,\,\,\,\text{for}\,\,|y|<\mu_j^{-1}.
\end{align*}
We denote $\bar{\xi}_j$ as spherical average of $\xi_j$ in the usual way. Then we have
\begin{align}
&P_{h_j}\xi_j(y)=\frac{n-4}{2}\bar{Q}\xi_j(y)^{\frac{n+4}{n-4}},\,\,\,\,|y|<\,\mu_j^{-1},\\
&|y|^{\frac{n-4}{2}}\xi_j(y)\leq C,\,\,\,\,|y|<\,\mu_j^{-1},\\
&\lim_{j\to\infty}\xi_j(0)=\infty,\\
&-\frac{4(n-1)}{n-2}\Delta_{h_j}\xi_j^{\frac{n-2}{n-4}}+R_{h_j}\xi_j^{\frac{n-2}{n-4}}\geq 0, \,\,\,\,|y|<\,\mu_j^{-1},\\
&r^{\frac{n-4}{2}}\bar{\xi}_j(r)\,\,\text{has precisely one critical point in}\,\,0< r <1,\\
&\label{criticalpoint}\frac{d}{d r}(r^{\frac{n-4}{2}}\bar{\xi}_j(r))=0\,\,\text{at}\,\,r=1.
\end{align}
Therefore $\{0\}$ is an isolated simple blowup point of the sequence $\{\xi_j\}$. Note that the Remark \ref{remark01} holds for $u_j$ so that
\begin{align}\label{ineqlowerbound1}
\xi_j(0)\xi_j(y)\geq C|y|^{4-n}\,\,\,\,\text{for}\,\,|y|\geq \mu_j^{-1}R_jM_j^{-\frac{2}{n-4}},
\end{align}
where $\mu_j^{-1}R_jM_j^{-\frac{2}{n-4}}\leq 1$. By Lemma \ref{lemH}, there exists $C>0$ independent of $j$ and $k$ so that for any $k\in \mathbb{R}$,
\begin{align}\label{ineqpolynomialgrowth}
\max_{2^k\leq |y|\leq 2^{k+1}}\xi_j(0)\xi_j(y)\,\,\leq\, C \min_{2^k\leq |y|\leq 2^{k+1}}\xi_j(0)\xi_j(y),\,\,\text{when}\,\,2^{k+1}<\mu_j^{-1}\frac{\delta_2}{3}.
\end{align}
Note that $Q_{h_j}\geq 0$ and $R_{h_j}>0$ in $M$. Also the metrics $h_j$ are all well controlled in $|y|\leq 1$. In proof of Lemma \ref{lemupperboundestimates1} the maximum principle holds for $h_j$ and the coefficients of the test function are still uniformly chosen for $h_j$ so that the estimate in Lemma \ref{lemupperboundestimates1} holds for each $\xi_j$ in $|y|\leq \tilde{\delta}_2$ for some $\tilde{\delta}_2<1$ independent of $j$. Similarly Proposition \ref{propupperbound} holds for $\xi_j$ in $|y|\leq\tilde{\delta}_2$. This combining with $(\ref{ineqlowerbound1})$ and $(\ref{ineqpolynomialgrowth})$ implies
\begin{align*}
C(K)^{-1}\leq \xi_j(0)\xi_j(y) \leq C(K)
\end{align*}
for $K\subset\subset \mathbb{R}^n-\{0\}$ when $j$ is large, $h_j$ converges to the flat metric and there exists $a>0$ so that $\xi_j(0)\xi_j(y)$ converges to
\begin{align*}
h(y)=a|y|^{4-n}+b(y)\,\,\text{in}\,\,C_{loc}^4(\mathbb{R}^n-\{0\}),
\end{align*}
where $b(y)\in C^4(\mathbb{R}^n)$ satisfies
\begin{align*}
\Delta^2 b =0
\end{align*}
in $\mathbb{R}^n$. Here $h>0$ in $\mathbb{R}^n-\{0\}$. Also,
\begin{align}\label{ineqscalarcurvature1}
-\Delta h(y)^{\frac{n-2}{n-4}}\geq 0,\,\,|y|>0.
\end{align}
Moreover, for a fixed point $y_0$ in $|y|=1$, by $(\ref{ineqpolynomialgrowth})$,
\begin{align*}
h(y)\leq |y|^{2+\frac{\ln C}{\ln 2}} h(y_0),
\end{align*}
for $|y|\geq 1$. Since $h>0$ for $|y|>0$, it follows that $b(y)$ is a polyharmonic function of polynomial growth on $\mathbb{R}^n$. Therefore, $b(y)$ must be a polynomial in $\mathbb{R}^n$, see \cite{Armitage}. Non-negativity of $h$ near infinity implies that $b(y)$ is of even order. Then either $b(y)$ is a non-negative constant or $b(y)$ is a polynomial of even order with order at least two and $b(y)$ is non-negative at infinity. The later case contradicts with $(\ref{ineqscalarcurvature1})$ for $y$ near infinity. Therefore, $b(y)$ must be a non-negative constant on $\mathbb{R}^n$ and
\begin{align*}
h(y)=a |y|^{4-n}+b
\end{align*}
with a constant $a>0$ and a constant $b$.

By $(\ref{criticalpoint})$,
\begin{align*}
\frac{d}{d r}(r^{\frac{n-4}{2}}h(r))=0\,\,\text{at}\,\,r=1.
\end{align*}
We then have $b=a>0$, which contradicts with Proposition \ref{propinequality}. In fact, Proposition \ref{propinequality} applies to isolated simple blowup points with respect the sequence of metrics $\{h_j\}$ with uniform curvature bound and uniform bound of injectivity radius with the property that $Q_{h_j}>0$ and $R_{h_j}>0$.
\end{proof}

\section{Compactness of solutions to the constant $Q$-curvature equations}
Based on Proposition \ref{propupperbound} and Proposition \ref{propisolatedsingularpoints}, proof of compactness of the solutions is more or less standard, see in \cite{Li-Zhu}. But again we need to deal with the limit of the blowup argument carefully, see Lemma \ref{lembehaviornearlargepoint} and Proposition \ref{propdistanceofsingularities}.

We first show that there are no bubble accumulations.

\begin{lem}\label{lembehaviornearlargepoint}
Let $(M^n, g)$ be a closed Riemannian manifold of dimension $5\leq n \leq 9$ with $R_g\geq 0$, and also $Q_g\geq 0$ with $Q_g(p_0)>0$ for some point $p_0\in M$. For any given $\epsilon>0$ and large constant $R>1$, there exist some constant $C_1>0$ depending on $M,\,g,\,\epsilon,\,R$, $\|Q_g\|_{C^1(M)}$ such that for any solution $u$ to $(\ref{equation1})$ and any compact subset $K \subset M$ satisfying
\begin{align*}
&\max_{p\in M-K}d(p, K)^{\frac{n-4}{2}}u(p)\geq C_1,\,\,\,\text{if}\,\,K\neq\emptyset,\,\,\text{and}\\
&\max_{p\in M}u(p)\geq C_1,\,\,\,\text{if}\,\,K=\emptyset,
\end{align*}
we have that there exists some local maximum point $p'$ of $u$ in $M-K$ with $B_{R\,u(p')^{-\frac{2}{n-4}}}(p') \subset M-K$ satisfying
\begin{align}\label{ineqlimitbehavior1}
\|u(p')^{-1}u(\exp_{p'}(u(p')^{-\frac{2}{n-4}}y))\,-\,(1+4^{-1}|y|^2)^{-\frac{n-4}{2}}\|_{C^4(|y|\leq 2R)}<\epsilon.
\end{align}

\end{lem}
\begin{proof}
We argue by contradiction. That is to say, there exist a sequence of compact subsets $K_j$ and a sequence of solutions $u_j$ to $(\ref{equation1})$ on $M$ such that
\begin{align*}
\max_{p\in M-K_j}d(p, K_j)^{\frac{n-4}{2}}u(p)\geq j,
\end{align*}
with $d(p, K_j)=1$ when $K_j=\emptyset$, but no point satisfies $(\ref{ineqlimitbehavior1})$. We choose $x_j \in M-K_j$ satisfying
\begin{align*}
d_g(x_j, K_j)^{\frac{n-4}{2}}u_j(x_j)=\max_{p\in M- K_j}d_g(p, K_j)^{\frac{n-4}{2}}u_j(p).
\end{align*}
We then define
\begin{align*}
v_j(y)=u_j(x_j)^{-1}u_j(\exp_{x_j}(u_j(x_j)^{-\frac{2}{n-4}}y)),\,\,\text{for}\,\,|y|\leq R_j=\frac{1}{4}u_j(x_j)^{\frac{2}{n-4}}d_g(x_j, K_j).
\end{align*}
Let $h_j=u_j(x_j)^{\frac{4}{n-4}}g$. The resecaled function $v_j$ satisfies
\begin{align}
P_{h_j}v_j=\frac{n-4}{2}\bar{Q}v_j^{\frac{n+4}{n-4}},
\end{align}
and by Theorem \ref{thm1},
\begin{align}
\Delta_{h_j}v_j\leq \frac{(n-4)}{4(n-1)}R_{h_j}v_j.
\end{align}
We will analyze limit of the sequence $\{v_j\}$ as in Theorem \ref{thmlowerbound} and conclude that $(\ref{ineqlimitbehavior1})$ indeed holds. By assumption,
\begin{align*}
R_j=\frac{1}{4}u_j(x_j)^{\frac{2}{n-4}}d_g(y_j, K_j)\geq \frac{1}{4}j^{\frac{2}{n-4}},
\end{align*}
and
\begin{align*}
d_g(\exp_{x_j}(u_j(x_j)^{-\frac{2}{n-4}}y), K_j)\geq \frac{1}{2}d_g(x_j, K_j),\,\,\text{for}\,\,|y|\leq R_j.
\end{align*}
It follows that
\begin{align*}
0< v_j(y)&=u_j(x_j)^{-1}u_j(\exp_{x_j}(u_j(x_j)^{-\frac{2}{n-4}}y))\\
&\leq u_j(x_j)^{-1} d_g(\exp_{x_j}(u_j(x_j)^{-\frac{2}{n-4}}y), K_j)^{-\frac{n-4}{2}}d_g(x_j, K_j)^{\frac{n-4}{2}}u_j(x_j)\\
&\leq 2^{\frac{n-4}{2}},\,\,\text{for}\,\,|y|\leq R_j.
\end{align*}
Standard elliptic estimates imply that up to a subsequence,
\begin{align*}
v_j\to v\,\,\text{in}\,\,C_{loc}^4(\mathbb{R}^n),
\end{align*}
with $v$ satisfying
\begin{align*}
&\Delta^2v=\frac{n-4}{2}\bar{Q}v^{\frac{n+4}{n-4}}\,\,\text{in}\,\,\mathbb{R}^n,\\
&v(0)=1,\,\,0\leq v \leq \,\,2^{\frac{n-4}{2}}\,\,\text{in}\,\,\mathbb{R}^n,\\
&\Delta v\leq 0,\,\,\text{in}\,\,\mathbb{R}^n.
\end{align*}
By strong maximum principle, $v>0$ in $\mathbb{R}^n$. Then by the classification theorem of C.S. Lin (\cite{Lin}),
\begin{align*}
v(y)=\big(\frac{\lambda}{1+4^{-1}\lambda^2|y-\bar{y}|^2}\big)^{\frac{n-4}{2}}\,\,\text{in}\,\,\mathbb{R}^n,
\end{align*}
with $v(0)=1$ and $v(y)\leq \lambda^{\frac{n-4}{2}}\leq 2^{\frac{n-4}{2}}$. Therefore, $|\bar{y}|\leq C(n)$ with $C(n)>0$ only depending on $n$. We choose $y_j$ to be the local maximum point of $v_j$ converging to $\bar{y}$. Then $p_j=\exp_{x_j}(u_j(x_j)^{-\frac{2}{n-4}}y_j)\in M-K_j$ is a local maximum point of $u_j$. We now repeat the blowup argument with $x_j$ replaced by $p_j$ and $u_j(x_j)$ replaced by $u_j(p_j)$ and obtain the limit
\begin{align*}
v(y)=(1+4^{-1}|y|^2)^{-\frac{n-4}{2}}\,\,\text{in}\,\,\mathbb{R}^n.
\end{align*}
Therefore, for large $j$, there exists $p_j\in M- K_j$ so that $(\ref{ineqlimitbehavior1})$ holds. This contradicts with the assumption. Therefore, the proof of the lemma is completed.
\end{proof}
\begin{lem}\label{lemalargefunction}
Let $(M^n, g)$ be a closed Riemannian manifold of dimension $5\leq n \leq 9$ with $R_g\geq 0$, and also $Q_g\geq 0$ with $Q_g(p_0)>0$ for some point $p_0\in M$. For any given $\epsilon>0$ and a large constant $R>1$, there exist some constants $C_1>0$ and $C_2>0$ depending on $M,\,g,\,\epsilon,\,R$, $\|Q_g\|_{C^1(M)}$ such that for any solution $u$ to $(\ref{equation1})$ with
\begin{align*}
\max_{p\in M}u(p)> C_1,
\end{align*}
there exists some integer $N=N(u)$ depending on $u$ and $N$ local maximum points $\{p_1,...,p_N\}$ of $u$ such that
\begin{enumerate}[label=\roman{*})]
\item  for $i\neq j$,
\begin{align*}
\overline{B_{\gamma_i}(p_i)}\bigcap \overline{B_{\gamma_j}(p_j)}=\emptyset,
\end{align*}
with $\gamma_j=Ru(p_j)^{-\frac{2}{n-4}}$ and $B_{\gamma_j}(p_j)$ the geodesic $\gamma_j$-ball centered at $p_j$, and
\begin{align}\label{ineqlimitbehavior2}
\|u(p_j)^{-1}u(\exp_{p_j}(u(p_j)^{-\frac{2}{n-4}}y))\,-\,(1+4^{-1}|y|^2)^{-\frac{n-4}{2}}\|_{C^4(|y|\leq 2R)}<\epsilon,
\end{align}
where $y=u(p_j)^{\frac{2}{n-4}}x$, with $x$ geodesic normal coordinates centered at $p_j$, and $|y|=\sqrt{(y^1)^2+..+(y^n)^2}$.

\item for $i < j$, $d_g(p_i, p_j)^{\frac{n-4}{2}}u(p_j)\geq C_1$, while for $p\in M$
\begin{align*}
d_g(p, \{p_1,..,p_n\})^{\frac{n-4}{2}}u(p)\leq C_2.
\end{align*}
\end{enumerate}
\end{lem}

\begin{proof}
We will use Lemma \ref{lembehaviornearlargepoint} and prove the lemma by induction. To start, we apply Lemma \ref{lembehaviornearlargepoint} with $K=\emptyset$. We choose $p_1$ to be a maximum point of $u$ and $(\ref{ineqlimitbehavior2})$ holds. Next we let $K=\overline{B_{\gamma_1}(p_1)}$.

Assume that for some $i_0\geq 1$, $i)$ in the lemma holds for $1\leq j\leq i_0$ and $1\leq i<j$, and also $d_g(p_i, p_j)^{\frac{n-4}{2}}u(p_j)\geq C_1$ with $p_j$ chosen as in Lemma \ref{lembehaviornearlargepoint} by induction.( This holds for $i_0=1$.) Then we let $K=\bigcup_{j=1}^{i_0}\overline{B_{\gamma_j}(p_j)}$. It follows that for $\epsilon>0$ small, for any $p$ such that $d_g(p, p_j)\leq 2\gamma_j$ with $1\leq j \leq i_0$, we have
\begin{align*}
d_g(p, \{p_1,..,p_{i_0}\})^{\frac{n-4}{2}}u(p)&\leq d_g(p, p_j)^{\frac{n-4}{2}}u(p)\leq 2 d_g(p, p_j)^{\frac{n-4}{2}} u(p_j)\\
&\leq 2 ( 2R u(p_j)^{-\frac{2}{n-4}})^{\frac{n-4}{2}}u(p_j)=2^{\frac{n-2}{2}}R^{\frac{n-4}{2}},
\end{align*}
and therefore, for $p\in \bigcup_{j=1}^{i_0}\overline{B_{2\gamma_j}(p_j)}$,
\begin{align}\label{ineqinball}
d_g(p, \{p_1,..,p_{i_0}\})^{\frac{n-4}{2}}u(p)\leq 2^{\frac{n-2}{2}}R^{\frac{n-4}{2}}.
\end{align}

If for all $p\in M$
\begin{align*}
d_g(p, \{p_1,..,p_{i_0}\})^{\frac{n-4}{2}}u(p)\leq C_1,
\end{align*}
the induction stops. Else, we apply Lemma \ref{lembehaviornearlargepoint}, and we denote $p_{i_0+1}$ as the local maximum point $y_0$ obtained in Lemma \ref{lembehaviornearlargepoint} so that
\begin{align*}
B_{R\,u(p_{i_0+1})^{-\frac{2}{n-4}}}(p_{i_0+1}) \subset M-K.
\end{align*}
Therefore, $i)$ in the lemma holds for $i_0+1$.  Also, by assumption, $d_g(p_j, p_{i_0+1})^{\frac{n-4}{2}}u(p_{i_0+1})> C_1$. By the same argument, $(\ref{ineqinball})$ holds for $i_0$ replaced by $i_0+1$. The induction must stop in a finite time $N=N(u)$, since $\int_M u^{\frac{2n}{n-4}}dV_g$ is bounded and that
\begin{align*}
\int_{B_{\gamma_j}(p_j)}u^{\frac{2n}{n-4}}dV_g
\end{align*}
is bounded below by a uniform positive constant. It is clear now that for $p \in M - \bigcup_{j=1}^NB_{\gamma_j}(p_j)$,
\begin{align*}
d(p, \{p_1,..,p_N\})^{\frac{n-4}{2}}u(p)\leq 2^{\frac{n-4}{2}}d(p, \bigcup_{j=1}^NB_{\gamma_j}(p_j))^{\frac{n-4}{2}}u(p)\leq 2^{\frac{n-4}{2}}C_1.
\end{align*}
By induction, $(\ref{ineqinball})$ holds for $i_0$ replaced by $N$. We set $C_2=2^{\frac{n-2}{2}}R^{\frac{n-4}{2}}+2^{\frac{n-4}{2}}C_1$. This proves the lemma.
\end{proof}
The next proposition rules out the bubble accumulations.
\begin{prop}\label{propdistanceofsingularities}
Let $(M^n, g)$ be a closed Riemannian manifold of dimension $n=5$ 
 with $R_g\geq 0$, and also $Q_g\geq 0$ with $Q_g(p_0)>0$ for some point $p_0\in M$. For $\epsilon>0$ small enough and a constant $R>1$ large enough, there exists $\gamma>0$ depending on $M, g, \epsilon, R,$ $\|R_g\|_{C^1(M)}$ and $\|Q_g\|_{C^1(M)}$ such that for any solution $u$ to $(\ref{equation1})$ with $\max_{p\in M}u(p)> C_1$, we have
\begin{align*}
d(p_i, p_j)\geq \gamma,
\end{align*}
for $1\leq i,\,j\leq N$ and $i\neq j$, where $N=N(u)$, $p_j=p_j(u)$, $p_i=p_i(u)$ and $C_1$ are defined in Lemma \ref{lemalargefunction}.
\end{prop}
\begin{proof}
Suppose the proposition fails, which implies that there exist $\epsilon>0$ small and $R>0$ large and a sequence of solutions $u_j$ to $(\ref{equation1})$ such that $\max_{p\in M}u_j(p)> C_1$ and
\begin{align*}
\lim_{j\to\infty}\min_{i\neq k}d(p_i(u_j), p_k(u_j))=0.
\end{align*}
We denote $p_{j,1}$ and $p_{j,2}$ to be the two points realizing minimum distance in $\{p_1(u_j),..,p_N(u_j)\}$ of $u_j$ constructed in Lemma \ref{lemalargefunction}. Let $\bar{\gamma}_j=d_g(p_{j,1}, p_{j,2})$. Since
\begin{align*}
B_{Ru_j(p_{1,j})^{-\frac{2}{n-4}}}(p_{1,j})\bigcap B_{Ru_j(p_{2,j})^{-\frac{2}{n-4}}}(p_{2,j})=\emptyset,
\end{align*}
we have that $u_j(p_{1,j})\to \infty$ and $u_j(p_{2,j})\to \infty$.

Let $x=(x^1,..,x^n)$ be geodesic normal coordinates centered at $p_{1, j}$, $y=\bar{\gamma}_j^{-1} x$, $\exp_{p_{1,j}}(x)$ be exponential map under the metric $g$. We define the scaled metric $h_j=\bar{\gamma}_j^{\frac{4}{n-4}}g$, and the rescaled function
\begin{align*}
v_j(y)=\bar{\gamma}_j^{\frac{2}{n-4}}u_j(\exp_{p_{1,j}}(\bar{\gamma}_jy)).
\end{align*}
It follows that $v_j$ satisfies $v_j>0$ in $|y|\leq \bar{\gamma}_j^{-1}r_0$ and that
\begin{align}
&\label{ineq7.1}P_{h_j}v_j(y)=\frac{n-4}{2}\bar{Q}v_j(y)^{\frac{n+4}{n-4}},\,\,\text{for}\,\,|y|\leq \bar{\gamma}_j^{-1}r_0,\\
&\label{ineq7.2}\Delta_{h_j}v_j\leq \frac{(n-4)}{4(n-1)}R_{h_j}v_j,\,\,\text{for}\,\,|y|\leq \bar{\gamma}_j^{-1}r_0,
\end{align}
where $r_0$ is half of the injectivity radius of $(M, g)$. We define $y_k=y_k(u_j)\in \mathbb{R}^n$ such that $\exp_{p_{1,j}}(\bar{\gamma}_jy_k)=p_k$ for the points $p_k(u_j)$. It follows that for $p_k\neq p_{1, j}$,
\begin{align*}
|y_k|\geq 1+o(1)
\end{align*}
with $o(1)\to 0$ as $j\to \infty$. Let $y_{2, j}\in \mathbb{R}^n$ be so that $p_{2,j}=\exp_{p_{1, j}}(\bar{\gamma}_jy_{2, j})$. Then
\begin{align*}
|y_{2, j}|\to 1\,\,\text{as}\,\,j\to \infty.
\end{align*}
It follows that there exists $\bar{y} \in \mathbb{R}^n$ with $|\bar{y}|=1$ such that up to a subsequence,
\begin{align*}
\bar{y}=\lim_{j\to \infty}y_{2,j}.
\end{align*}
By Lemma \ref{lemalargefunction},
\begin{align*}
\bar{\gamma}_j\geq C \max\{Ru_j(p_{1,j})^{-\frac{2}{n-4}},\,Ru_j(p_{2,j})^{-\frac{2}{n-4}}\}.
\end{align*}
Therefore, we have
\begin{align*}
&v_j(0)\geq C_3,\,\,v_j(y_{2,j})\geq C_3\,\,\text{for some}\,\,C_3>0\,\,\text{independent of}\,\,j,\\
&y_k\,\,\text{is a local maximum point of }\,\,v_j\,\,\text{for all}\,\,1\leq k\leq N(u_j),\\
&\min_{1\leq k\leq N(u_j)}|y-y_k|^{\frac{n-4}{2}}v_j(y)\leq C_2\,\,\text{for all}\,\,|y|\leq \bar{\gamma}_j^{-1}.
\end{align*}
We {\bf claim} that
\begin{align}
v_j(0)\to \infty,\,\,\text{and}\,\,v_j(y_{2,j})\to \infty.
\end{align}
To see this, we first assume that one of them tends to infinity up to a subsequence, say $v_j(0)\to \infty$ for instance. It is clear that $0$ is an isolated blowup point, and by Proposition \ref{propisolatedsingularpoints} it is an isolated simple blowup point. Then $v_j(y_{2,j})\to \infty$ in this subsequence since otherwise, by the control $(\ref{ineqlimitbehavior2})$ at $p_{2, j}$ in Lemma \ref{lemalargefunction} and the rescaling, $v_j$ is uniformly bounded in a uniform neighborhood of $y_{2,j}$ and therefore by Harnack inequality $(\ref{ineqH})$ and Proposition \ref{propupperbound}, $v_j\to 0$ near $p_{2,j}$, contradicting with $v_j(y_{2,j})\geq C_3$. If both $v_j(0)$ and $v_j(y_{2,j})$ are uniformly bounded, similar argument shows that $v_j$ is uniformly bounded on any fixed compact subset of $\mathbb{R}^n$. Then as discussed in Lemma \ref{lembehaviornearlargepoint}, $v_j\to v$ in $C_{loc}^4(\mathbb{R}^n)$ with $v>0$ and
\begin{align*}
\Delta^2v=\frac{n-4}{2}\bar{Q}v^{\frac{n+4}{n-4}}
\end{align*}
 in $\mathbb{R}^n$. Also, $0$ and $\bar{y}$ are local maximum points of $v$. That contradicts with the classification theorem in \cite{Lin}. The {\bf claim} is established.
 Therefore, both $0$ and $\bar{y}$ are isolated simple blowup points of $v_j$. Let $K_0$ be the set of blowup points of $\{v_j\}$ after passing to a subsequence. It is clear that $0, \bar{y}\in K_0$ and for any two distinct points $y, z \in K$, $d_g(y, z)\geq 1$. By Proposition \ref{propupperbound}, $v_j(0)v_j$ is uniformly bounded in any fixed compact subset of $\mathbb{R}^n-K_0$. Multiplying $v_j(0)$ on both sides of $(\ref{ineq7.1})$ and $(\ref{ineq7.2})$, we have that up to a subsequence,
 \begin{align*}
 \lim_jv_j(0)v_j\to F\geq 0\,\,\text{in}\,\,C_{loc}^4(\mathbb{R}^n-K_0),
 \end{align*}
 such that
 \begin{align}
 &\label{ineq7-a}\Delta^2F=0,\,\,\text{in}\,\,\mathbb{R}^n-K_0,\\
 &\label{ineq7-b}\Delta F\leq 0,\,\,\text{in}\,\,\mathbb{R}^n-K_0.
 \end{align}
 Since all the blowup points in $K_0$ are isolated simple blowup points, by Proposition \ref{propupperbound},
 \begin{align*}
 F(y)=a_1|y|^{4-n}+\Phi_1(y)=a_1|y|^{4-n}+a_2|y-\bar{y}|^{4-n}+\Phi_2(y)
 \end{align*}
 for $y \in \mathbb{R}^n-K_0$ with the constants $a_1, a_2>0$. Moreover, $\Phi_2\in C^4(\mathbb{R}^n-(K_0-\{0, \bar{y}\}))$ and $\Phi_2$ satisfies $(\ref{ineq7-a})$ in $\mathbb{R}^n-(K_0-\{0, \bar{y}\})$. We define $\xi= \Delta \Phi_1$ in $\mathbb{R}^n-(K_0-\{0\})$. By $(\ref{ineq7-b})$, $F>0$ in $\mathbb{R}^n-K_0$. Therefore,
 \begin{align}
 &\label{ineq71a}\liminf_{|y|\to \infty}\Phi_1(y)=\liminf_{|y|\to \infty} (F(y)-a_1|y|^{4-n})\geq 0,\\
 &\liminf_{|y|\to \infty}\xi(y)=\liminf_{|y|\to \infty}\Delta (F(y)-a_1|y|^{4-n})\leq 0.
 \end{align}
 Moreover, $\xi<0$ near any isolated point in $\mathbb{R}^n-(K_0-\{0\})$ by Proposition \ref{propupperbound}. Applying strong maximum principle to $\xi$ and the equation
 \begin{align*}
 \Delta \xi=\Delta^2(F-a_1|y|^{4-n})=0
 \end{align*}
 in $\mathbb{R}^n-(K_0-\{0\})$, we have that
 \begin{align*}
 \xi=\Delta \Phi_1<0
 \end{align*}
 in $\mathbb{R}^n-(K_0-\{0\})$. Since $\Phi_1>0$ near any isolated point in $\mathbb{R}^n-(K_0-\{0\})$ by Proposition \ref{propupperbound}, and also $(\ref{ineq71a})$ holds, applying strong maximum principle to $\Phi_1$ and $\Delta \Phi_1<0$ in $\mathbb{R}^n-(K_0-\{0\})$, we have that $\Phi_1>0$ in $\mathbb{R}^n-(K_0-\{0\})$.
 It follows that
 \begin{align*}
 F(y)=a_1|y|^{4-n}+\Phi_1(0)+O(|y|)\,\,\text{with}\,\,\Phi_1(0)>0\,\,\text{near}\,\,y=0,
 \end{align*}
 contradicting with Proposition \ref{propinequality}.(It is easy to check that Proposition \ref{propinequality} applies for the scaled metrics $h_j$ instead of $g$.) Proposition \ref{propdistanceofsingularities} is then established.
\end{proof}
We are now ready to prove the compactness theorem of positive solutions to the equation $(\ref{equation1})$.

\begin{proof}[Proof of Theorem \ref{thm27}]
By Lemma \ref{lemlowerbound1} and ellipticity theorem for $(\ref{equation1})$, we only need to show that there is a constant $C>0$ depending on $M$ and $g$ such that
\begin{align*}
u\leq C.
\end{align*}
Suppose the contrary, then there exists a sequence of positive solutions $u_j$ to $(\ref{equation1})$ such that
\begin{align*}
\max_{p\in M}u_j\to \infty
\end{align*}
as $j\to \infty$. By Proposition \ref{propdistanceofsingularities}, after passing to a subsequence, there exists $N$ isolated simple blowup points $p_{1, j}\to p_1,\,...,\,p_{N, j}\to p_N$ with $N\geq 1$ independent of $j$. Applying Proposition \ref{propupperbound}, we have that up to a subsequence,
\begin{align*}
u_j(p_{1,j})u_j(p)\to F(p)=\sum_{k=1}^Na_kG_g(p_k, p)+b(p)\,\,\text{in}\,\,C_{loc}^4(M-\{p_1,..,p_N\}),
\end{align*}
where $a_1>0,\, ... ,\,a_N>0$ are  some constants, $G_g$ is Green's function of $P_g$ under the metric $g$ and $b(p)\in C^4(M)$ satisfying
\begin{align*}
P_gb=0
\end{align*}
in $M$. Since $Q_g\geq 0$ on $M$ with $Q_g>0$ at some point, by the strong maximum principle of $P_g$, $b=0$ in $M$. We know that $G_g(p_k, p)>0$ for $1\leq k\leq N$ by Theorem \ref{thm1}. Let $x=(x^1,..,x^n)$ be conformal normal coordinates( see \cite{Lee-Parker}) centered at $p_{1,j}$( resp. $p_1$) with respect to the conformal metric $h_j=\phi_j^{-\frac{4}{n-4}}g$(resp. $h=\phi^{-\frac{4}{n-4}}g$) such that
\begin{align*}
\det(h_{ij})=1+O(|x|^{10n}).
\end{align*}
Then there exists $C_1>0$ independent of $j$ such that
\begin{align*}
C_1^{-1}\leq \phi_j\leq C_1,
\end{align*}
and
\begin{align*}
\|\phi_j-\phi\|_{C^5(M)}\to 0\,\,\text{as}\,\,j\to \infty.
\end{align*}
As shown in Theorem \ref{thm1}, under the conformal normal coordinates $x$ centered at $p_1$, the Green's function under metric $h$ satisfies
\begin{align*}
G_h(p_1, p)=\phi^2(p)G_g(p_1, p)=d_h(p_1, p)^{4-n}+A+o(1)
\end{align*}
near $p_1$ with the constant $A>0$ and $o(1)\to 0$ as $p\to p_1$. Therefore,
 \begin{align*}
 \phi(p)^2F(p)=a_1d_h(p_1, p)^{4-n}+B+o(1)
 \end{align*}
 with $B=a_1A+\sum_{k=2}^Na_k\phi(p_1)^2G_g(p_k, p_1)>0$ and $o(1)\to 0$ as $p\to p_1$. Note that since $\phi_j$ are uniformly controlled in the construction of conformal normal coordinates and the corresponding metrics, the conclusions in Corollary \ref{corerrorterms} and consequently in Proposition \ref{propinequality} still hold for $g$ replaced by the conformal metrics $h_j$ and $u_j$ replaced by $\tilde{u}_j=\phi_j u_j$. This leads to a contradiction. Therefore, Theorem \ref{thm27} is established.
\end{proof}

\section{Appendix: Positive solutions of certain linear fourth order elliptic equations in punctured balls}
Assume $B_{\delta}(\bar{x})$ is a geodesic $\delta$-ball on $\mathbb{R}^n$ under the metric $g$ with $2\delta$ less than the injectivity radius. For application, for $5\leq n \leq 9$ it could sometime be assumed as a geodesic $\delta$-ball embedded in a closed Riemannian manifold $(M^n, g)$, where $(M, g)$ is as in Proposition \ref{propupperbound}.
\begin{lem}\label{lemlinear9.1}
Let $u \in C^4(B_{\delta}(\bar{x})-\{\bar{x}\})$ be a solution to
\begin{align}\label{equlinear}
P_g u=0\,\,\text{in}\,\,B_{\delta}(\bar{x})-\{\bar{x}\}.
\end{align}
If $u(p)=o(d_g(p, \bar{x}))^{4-n}$ as $p\to \bar{x}$, then $u\in C_{loc}^{4, \alpha}(B_{\delta}(\bar{x}))$ for $0<\alpha <1$.
\end{lem}
\begin{proof} The proof is standard.

\noindent Step 1. We show that $(\ref{equlinear})$ holds in $B_{\delta}(\bar{x})$ in distribution sense.

To see this, given any small $\epsilon>0$, we define the cutoff function $\eta_{\epsilon}$ on $B_{\delta}(\bar{x})$ with $0<\eta_{\epsilon}<1$ so that
\begin{align*}
&\eta_{\epsilon}(p)=1\,\,\,\,\,\,\text{for}\,\,d_g(p, \bar{x})\leq \epsilon,\\
&\eta_{\epsilon}(p)=0\,\,\,\,\,\,\text{for}\,\,d_g(p, \bar{x})\geq 2\epsilon,\\
&|\nabla\eta_{\epsilon}(p)|\leq C\epsilon^{-1}\,\,\text{for}\,\,\epsilon\leq d_g(p, \bar{x})\leq 2\epsilon.
\end{align*}
For any given $\phi \in C_c^{\infty}(B_{\delta}(\bar{x}))$ we multiply $\phi (1-\eta_{\epsilon})$ on both side of $(\ref{equlinear})$ and do integration by parts,
\begin{align*}
\int_{B_{\delta}(\bar{x})}P_g(\phi (1-\eta_{\epsilon}))u dV_g=0.
\end{align*}
Let $\epsilon\to 0$, then
\begin{align*}
\int_{B_{\delta}(\bar{x})}P_g\phi (1-\eta_{\epsilon})u dV_g=O(1)(C\epsilon^{-4}\int_{B_{2\epsilon}(\bar{x})-B_{\epsilon}(\bar{x})}|u|)+C\int_{B_{\epsilon}(\bar{x})}|u|\to 0,
\end{align*}
where in the last step we have used $u(p)=o(d_g(p, \bar{x}))^{4-n}$. Therefore, Step 1 is established.

\noindent Step 2. The assumption of $u$ near $\bar{x}$ implies that $u\in L^p_{loc}(B_{\delta}(\bar{x}))$ for any $1<p<\frac{n}{n-4}$. By $W^{4, p}$ estimates of the elliptic equation we obtain that $u\in W_{loc}^{4,p}(B_{\delta}(\bar{x}))$, see \cite{S.Agmon} for instance. Then standard bootstrap argument gives $u\in C_{loc}^{4, \alpha}(B_{\delta}(\bar{x}))$.

\end{proof}
For later use, we now employ Lemma 9.2 from \cite{Li-Zhu} without proof.
\begin{lem}\label{lemGreenfunction}
There exists some constant $0<\delta_0\leq \delta$ depending on $n,\,\|g_{ij}\|_{C^2(B_{\delta}(\bar{x}))}$ and $\|R_g\|_{L^{\infty}(B_{\delta}(\bar{x}))}$ such that the maximum principle for $-\frac{4(n-1)}{n-2}\Delta_g+R_g$ holds on $B_{\delta_0}(\bar{x})$, and there exists a unique $G_1(p)\in C^2(B_{\delta_0}(\bar{x})-\{\bar{x}\})$ satisfying
\begin{align*}
&-\frac{4(n-1)}{n-2}\Delta_gG_1+R_gG_1=0\,\,\text{in}\,\,B_{\delta_0}(\bar{x})-\{\bar{x}\},\\
&G_1=0\,\,\text{on}\,\,\partial B_{\delta_0}(\bar{x}),\\
&\lim_{p\to \bar{x}}d_g(p, \bar{x})^{n-2}G_1(p)=1.
\end{align*}
Furthermore, $G_1(p)=d_g(p, \bar{x})^{2-n}+\mathcal {R}(p)$ where $\mathcal {R}(p)$ satisfies for all $0<\epsilon <1$ that
\begin{align*}
d_g(p, \bar{x})^{n-4+\epsilon}|\mathcal {R}(p)|+ d_g(p, \bar{x})^{n-3+\epsilon}|\nabla \mathcal {R}(p)|\leq C(\epsilon),\,\,p\in\,B_{\delta_0}(\bar{x}),\,\,n\geq 4,
\end{align*}
where $C(\epsilon)$ depends on $\epsilon$, $n$, $\|g_{ij}\|_{C^2(B_{\delta}(\bar{x}))}$ and $\|R_g\|_{L^{\infty}(B_{\delta}(\bar{x}))}$.
\end{lem}
\begin{lem}\label{lem9.3}
Suppose a positive function $u\in C^4(B_{\delta}(\bar{x})-\{\bar{x}\})$ satisfies $(\ref{equlinear})$ in $B_{\delta}(\bar{x})-\{\bar{x}\}$, and assume that there exists a constant $C>0$ such that for $0< r < \delta$, the Harnack inequality holds:
\begin{align*}
\max_{d_g(p, \bar{x})=r}u(p)\leq C \min_{d_g(p, \bar{x})=r}u(p).
\end{align*}
If moreover,
\begin{align*}
-\frac{4(n-1)}{n-2}\Delta_gu^{\frac{n-2}{n-4}}+R_gu^{\frac{n-2}{n-4}}\geq 0\,\,\text{in}\,\,B_{\delta}(\bar{x})-\{\bar{x}\},
\end{align*}
then
\begin{align*}
a=\limsup_{p\to \bar{x}} d_g(p, \bar{x})^{n-4}u(p)<+\infty.
\end{align*}
\end{lem}
\begin{proof}
If the lemma is not true, then for any $A>0$, there exists $r_i\to 0^+$ satisfying
\begin{align*}
u(p)>\,A\,r_i^{4-n},\,\,\text{for all}\,\,d_g(p, \bar{x})=r_i.
\end{align*}
Let $v_A=\frac{A^{\frac{n-2}{n-4}}}{2}G_1$ with $G_1$ in Lemma \ref{lemGreenfunction}. For $i$ large, by maximum principle,
\begin{align*}
u(p)^{\frac{n-2}{n-4}}\geq v_A(p)\,\,\text{for}\,\,r_i <d_g(p, \bar{x})<\delta_0.
\end{align*}
As $i\to \infty$, it holds that
\begin{align*}
u(p)^{\frac{n-2}{n-4}}\geq v_A(p)\,\,\,\,\text{for}\,\,0<d_g(p, \bar{x})<\delta_0.
\end{align*}
Since $A$ can be arbitrarily large, $u(p)=\infty$ in $0<d_g(p, \bar{x})<\delta_0$, which is a contradiction.
\end{proof}
\begin{prop}\label{propsingularity}
Let $u$ be as in Lemma \ref{lem9.3}. Then there exists a constant $b\geq 0$ such that
\begin{align}\label{equsingularity9.1}
u(p)=bG(p, \bar{x})+E(p)\,\,\text{for}\,\,p\in B_{\delta_0}(\bar{x})-\{\bar{x}\},
\end{align}
where $G$ is Green's function of $P_g$, ( for the existence of the Green's function in our application, it is limit of Green's function of Paneitz operator of a sequence of metrics on $M$ restricted to certain domains, and when $g$ is the flat metric, let $G(x,y)=c_n|x-y|^{4-n}$) and $\delta_0$ is defined in Lemma \ref{lemGreenfunction}. Here $E\in C^4(B_{\delta_0}(\bar{x}))$ satisfies $P_g E=0$ in $B_{\delta_0}(\bar{x})$.
\end{prop}
\begin{proof}

We rewrite $(\ref{equlinear})$ as
\begin{align*}
\Delta_g(\Delta_g u)= \text{div}_g(a_nR_g g -b_nRic_g)\nabla_gu - \frac{n-4}{2}Q_gu.
\end{align*}
By Lemma \ref{lem9.3}, $0 < u(p) \leq a_1 G(p, \bar{x})$ with some constant $a_1>a$ in $B_{\delta_0}(\bar{x})-\{\bar{x}\}$ with $\delta_0>0$ in Lemma \ref{lemGreenfunction}. Combining with the interior estimates, there exists a constant $C>0$ such that
\begin{align}\label{ineqboundlowerorderterms}
&|\text{div}_g(a_nR_g g -b_nRic_g)\nabla_gu - \frac{n-4}{2}Q_gu|\leq C d_g^{2-n}(p, \bar{x}),\,\,\text{and}\\
&\label{ineqLaplacian}|\Delta_g u(p)|\leq C\,d_g^{2-n}(p, \bar{x}),
\end{align}
for $p\in \overline{B}_{\delta_0}(\bar{x})-\{0\}$. We define $G_2$ to be a Green's function of $\Delta_g$ on $\overline{B}_{\delta_0}(\bar{x})$ such that
\begin{align}\label{inequkernel0}
0< G_2(p, q)\leq C d_g(p, q)^{2-n},
\end{align}
for some constant $C>0$ and any two distinct points $p$ and $q$ in $B_{\delta_0}(\bar{x})$.
Then
\begin{align*}
\phi_1(p)=\int_{B_{\delta_0}(\bar{x})}G_2(p, q)(\text{div}_g(a_nR_g g -b_nRic_g)\nabla_gu(q) - \frac{n-4}{2}Q_gu(q)) d V_g(q)
\end{align*}
is a special solution to the equation
\begin{align*}
\Delta_g \phi= \text{div}_g(a_nR_g g -b_nRic_g)\nabla_gu - \frac{n-4}{2}Q_gu,\,\,\text{in}\,\,B_{\delta_0}(\bar{x})-\{\bar{x}\}.
\end{align*}
Combining $(\ref{ineqboundlowerorderterms})$ and $(\ref{inequkernel0})$, we have that there exists a constant $C>0$ such that
\begin{align*}
|\phi_1(p)|\leq C d_g(p, \bar{x})^{4-n},
\end{align*}
for $p \in B_{\delta_0}(\bar{x})-\{\bar{x}\}$. Therefore,
\begin{align*}
\Delta_g(\Delta_g u- \phi_1)=0,\,\,\text{in}\,\,B_{\delta_0}(\bar{x})-\{\bar{x}\}.
\end{align*}
Since we also have $(\ref{ineqLaplacian})$, proof of Proposition 9.1 in \cite{Li-Zhu} applies and there exists a constant $-C\leq b_2\leq C$ such that
\begin{align*}
(\Delta_g u(p)- \phi_1(p))=b_2 G_1(p)+\varphi_1(p),\,\,\text{in}\,\,B_{\delta_0}(\bar{x})-\{\bar{x}\},
\end{align*}
with $G_1$ as in Lemma \ref{lemGreenfunction} and $\varphi_1$ a harmonic function on $\overline{B}_{\delta_0}(\bar{x})$. Therefore,
\begin{align*}
\Delta_g u(p)=b_2 G_1(p)+\phi_1(p)+\varphi_1(p),\,\,\text{in}\,\,B_{\delta_0}(\bar{x})-\{\bar{x}\}.
\end{align*}
By the same argument, there exists $ b_3 \in \mathbb{R}$ such that
\begin{align*}
u(p)&=b_3G_1(p)+\phi_2(p)+\int_{B_{\delta_0}(\bar{x})}G_2(p, q)[b_2 G_1(q)+\phi_1(q)+\varphi_1(q)] d V_g(q)\\
&=b_3G_1(p)+\phi_2(p)+O(d_g(p, \bar{x})^{4-n})
\end{align*}
in $B_{\delta_0}(\bar{x})-\{\bar{x}\}$, with $\varphi_2$ a harmonic function on $B_{\delta_0}(\bar{x})$. But since $0 < u(p) \leq a_1 G(p, \bar{x})$, we have $b_3=0$ and
\begin{align*}
u(p)=b_2\int_{B_{\delta_0}(\bar{x})}G_2(p, q) G_1(q) d V_g(q)+o(d_g(p, \bar{x})^{4-n})
\end{align*}
in $B_{\delta_0}(\bar{x})-\{\bar{x}\}$. Therefore, there exists a constant $b\geq 0$ such that
\begin{align*}
u(p)&=bd_g(p, \bar{x})^{4-n}+o(d_g(p, \bar{x})^{4-n})\\
&=bG(p, \bar{x})^{4-n}+o(d_g(p, \bar{x})^{4-n}).
\end{align*}
Then by Lemma \ref{lemlinear9.1}, there exists a function $E \in C^4(B_{\delta_0}(\bar{x}))$ satisfying $(\ref{equlinear})$ and
\begin{align*}
u(p)=bG(p, \bar{x})^{4-n}+E(p)
\end{align*}
for $p\in B_{\delta_0}(\bar{x})-\{\bar{x}\}$.

 This completes the proof of the proposition.
\end{proof}
Using Proposition \ref{propsingularity}, we immediately conclude the following corollary.

\begin{cor}\label{cor9.1}
For $n\geq 5$, assume that $u\in C^4(B_{\delta_0}(\bar{x})-\{\bar{x}\})$ is a positive solution of $(\ref{equlinear})$ with $\bar{x}$ a 
 singular point, and also the assumptions in Lemma \ref{lem9.3} holds for $u$. Then
\begin{align*}
\lim_{r\to 0}\int_{B_r(\bar{x})}(P_g u- \frac{n-4}{2}\bar{Q} u) d V_g&=\lim_{r\to 0}\int_{\partial B_r(\bar{x})}\big(\frac{\partial}{\partial \nu}\Delta_g u-(a_nR_g\frac{\partial}{\partial \nu} u-b_nRic_g(\nabla_g u, \nu))\big) ds_g\\
&=b\,\lim_{r\to 0}\int_{\partial B_r(\bar{x})}\frac{\partial}{\partial \nu}\Delta_gG(p, \bar{x}) ds_g(p)=2(n-2)(n-4)|\mathbb{S}^{n-1}|\,b>0,
\end{align*}
where $\nu$ is the outer unit normal and $b>0$ is as in $(\ref{equsingularity9.1})$.
\end{cor}

 \end{document}